\documentclass[11pt,twoside]{article}

\usepackage{mathrsfs,amsfonts,amsmath,amssymb}
\usepackage{latexsym}
\usepackage{comment} 
\usepackage{enumitem}
\newtheorem{satz}{Theorem}
\newtheorem{proposition}[satz]{Proposition}
\newtheorem{theorem}[satz]{Theorem}
\newtheorem{lemma}[satz]{Lemma}
\newtheorem{definition}[satz]{Definition}
\newtheorem{corollary}[satz]{Corollary}
\newtheorem{remark}[satz]{Remark}
\newtheorem{exm}[satz]{Example}

\def\Z{\mathbb {Z}}
\def\F{\mathbb {F}}
\def\E{\mathsf{E}}

\def\a{\alpha}

\def\P{{\cal P}}

\def\d{\delta}
\def\o{\omega}
\def\({\big (}
\def\){\big )}
\def\g{\gamma}
\def\G{\Gamma}

\def\dim{{\rm dim}}
\def\le{\leqslant}
\def\ge{\geqslant}
\def\_phi{\varphi}
\def\eps{\varepsilon}

\def\Gr{{\mathbf G}}
\def\FF{\widehat}

\def\ov{\overline}

\def\Span{{\rm Span\,}}

\def\la{\lambda}
\def\D{\Delta}

\def\T{\mathsf{T}}

\def\F{\mathbb {F}}
\def\R{{\mathbb R}}

\def\bp{\bigskip}

\newenvironment{dedication}
  {%\clearpage           % we want a new page          %% I commented this
   \thispagestyle{empty}% no header and footer
   \itshape             % the text is in italics
   \raggedleft          % flush to the right margin
  }
  {\par % end the paragraph
  }

\author{Shkredov I.D.}
\title{Additive dimension and the growth of sets 
}
\date{}
\begin{document}
	\maketitle

  \begin{dedication}
    Dedicated to 
    %my PhD student
%    \par   
%%    \vspace{2\baselineskip}
    Konstantin I. Olmezov\\ 
    (01.12.1995---20.03.2022), \\
     the victim of this heinous war
  \end{dedication}

%\begin{document}

\begin{center}
	Annotation
\end{center}

{\it \small
    We develop the theory of the additive dimension $\dim (A)$, i.e. the size of a maximal dissociated subset of a set $A$. %$A$ belongs to an abelian group $\Gr$. 
    It was shown that the additive dimension is closely connected with the growth of higher sumsets $nA$ of our set $A$. We apply 
    %our method 
    this approach 
    to demonstrate  that for any small  multiplicative subgroup $\G$ the sequence $|n\G|$  grows 
    very fast. 
    %almost optimally.
    Also, we obtain a series of applications to the sum--product phenomenon and to the Balog--Wooley decomposition--type results. 
}
\\

\section{Introduction}
\label{sec:introduction}

\subsection{General results on dimensions and the growth}

Let $\Gr$ be an abelian group and $k$ be a positive integer. A finite set $\Lambda \subseteq \Gr$ is called {\it $k$--dissociated} if any  equality of the form 
\[
    \sum_{\la \in \Lambda} \eps_\la \la = 0 \,,
    \quad \quad 
    \mbox{ where }
    \quad \quad 
    |\eps_\la| \le k\,, \quad \quad  \forall \la \in \Lambda 
\]
implies $\eps_\la = 0$ for all $\la$. 
If $k=1$, then $\Lambda$ is called {\it dissociated}. 
Let $\dim_k (A)$ be the size of the largest $k$--dissociated subset of $A$ and we call $\dim(A):= \dim_1 (A)$ the {\it additive dimension} of $A$. 
%Write $\dim (A)$ for $\dim_1 (A)$. 
Clearly, $\dim_k (A) \ge \log_{2k+1} |A|$ and for $A=\{1,2, \dots, n\}$, say, one has $\dim_k (A) \ll \log_{2k+1} |A|$. 
One of the main ideas of this paper is to treat the ratios  (below $k$ is a large parameter)
\[
    \frac{\dim (A)}{\log |A|}
    \quad \quad 
        \mbox{and}
    \quad \quad 
    \frac{\dim_k (A)}{\log_{2k+1} |A|} \,,
\]
as some measures,  
%of randomness of our set $A$ 
which control the growth of our set $A$
see, e.g, inequality \eqref{f:growth_nA_intr1} of Theorem  \ref{t:growth_nA_intr}.
%below.  
The notions of dissociativity and dimensions  appeared naturally in analysis, see \cite{Rudin_book,KS}, as well as in additive combinatorics, see \cite{BK_AP3,BK_struct,BS,Bourgain_AP3,Bourgain_AP3_new,c1,KS_random,sanders-log,Sanders_1,sy}. 
Previously (see \cite{c1}, \cite{KS}, \cite{ss_dim}, \cite{sy}), we studied the case when the quantity $\frac{\dim (A)}{\log |A|}$ is small. Namely, if a set $A$ is additively rich in a sense (e.g., $|A+A| \ll |A|$), then $\dim(A)$ is small comparable to $\log |A|$ and thus this case can be treated as the ``structural'' one. 
Even the famous 
%additive--combinatorial  
Polynomial Freiman--Ruzsa conjecture \cite{Green_FR} from the  structural theory of sets addition   demands about the possibility   to find 
a large subset $A_*$ of any set $A$ with 
small 
%doubling 
$|A+A|$
%sumset 
such that  $A_*$ has small dimension $\dim (A_*)$.
%is small. 
%is treated as additively rich   

In this paper we consider the opposite situation, studying the case when $\dim (A)$ is {\it large}. Moreover, instead of developing the theory of sets with small sumset $|A+A|$ we are interested in higher sumsets $nA$ for large $n$. 
We show that 
%(a variant of) 
the 
additive dimension
%(see rigorous formulation in section \ref{sec:}) 
controls higher sumsets, 
%and 
as well as 
higher energies
(see rigorous formulation in Section \ref{sec:definitions}). 
This phenomenon continues the classical line of additive combinatorics which is connected with the Freiman Lemma \cite{Freiman_book}, \cite[Lemma 5.13]{TV} (also, see recent achievements  in \cite{ZP}). Nevertheless, in the Freiman Lemma we have another dimension, which is defined in terms Freiman's isomorphisms, see \cite[Section 5.3]{TV}.

Our basic setting is the following. 
For any set $A \subseteq \Gr$ consider the growing sequence 
\begin{equation}\label{f:A_growth}
    |A| \le |2A| \le \dots \le |nA| \le \dots
\end{equation}
For example, if one 
%considers 
takes 
a set $A$ with  small doubling, then  sequence \eqref{f:A_growth} is somehow trivial: all sets $nA$ have comparable sizes with $A$. 
%is characterized by the 
In general, 
%case, 
sequence \eqref{f:A_growth} can be rather complex. 
Our first result (it is a combination of Lemmas \ref{l:dL}, \ref{l:LY} below) shows that any set $A$ has three stages of its ``life'', i.e. there are three basic lower bounds for the sequence $|nA|$ in terms of some  dimensions (and it is possible that $A$ does not grow at each stage).

\begin{theorem}
    Let $\Gr$ be an abelian group and $A\subseteq \Gr$ be a set.
    Then there is an absolute constant $C>0$ such that for any $n$ with 
    %$m \le \dim(A)/2$  
    $n\le C^{-1} \log |A|$
    one has 
\begin{equation}\label{f:growth_nA_intr1}
    |nA| \ge |A| \cdot \left( \frac{\dim (A)}{C \log |A|} \right)^{n-1} \,. 
\end{equation}
    Now if $C^{-1} \log |A| \le n\le \dim(A)/4$, then  
\begin{equation}\label{f:growth_nA_intr2}
    |nA| \ge \left( \frac{\dim (A)}{4n} \right)^{n-1} \,. 
\end{equation}
    Finally,  for $k =  \dim (A) \log \dim (A)$ one has 
\begin{equation}\label{f:growth_nA_intr3}
    |\dim^2_k (A) \log \dim_k (A)\, A| \ge 
    \exp (C^{-1} \dim_k (A) \cdot \log \dim_k (A) ) \,.
\end{equation}
\label{t:growth_nA_intr}
\end{theorem}

Thus any set $A$ grows if there is a nontrivial lower bound for $\dim (A)$, see stages \eqref{f:growth_nA_intr1} and \eqref{f:growth_nA_intr2}. After that further growth is possible if for a certain $k$  another dimension $\dim_k (A)$ can be estimated non--trivially in a sense, see bound  \eqref{f:growth_nA_intr3},  as well as Lemma \ref{l:dim_compare+}, where the estimate  $\dim_k (A) \cdot \log (k\dim_k (A)) \gg \dim (A)$ was obtained.

%{\bf zoo }

Our second structural result allows us to describe all sets such that sequence \eqref{f:A_growth} stops 
%growing 
after some steps in a sense that $|nA| \ll_n |A|^C$ for a certain constant $C\ge 1$. This is an important class of sets and previously  the author considered in detail just the case $n=2$, i.e. the family of sets $A$ with $|nA| \ll_n |2A|$ (we mention paper \cite{s_ineq} as  the first one in this direction). 
Combining  Corollary \ref{c:Q(A)_A} and Theorem \ref{t:A^C_growth} below, we obtain the following result. Given a set $A = \{a_1,\dots, a_n\} \subseteq \Gr$, recall that the combinatorial cube $Q(A)$ is $\{0,1\}\cdot a_1 + \dots + \{0,1\} \cdot a_n$. 

\begin{theorem}
    Let $\mathcal{R}$ be a ring, $A\subseteq \mathcal{R}$ be a set, $k=\dim (A) \log \dim (A)$, $K\ge 1$ be a real number 
    and $\Lambda_k$ be a maximal $k$--dissociated subset of $A$. 
    Suppose that all numbers $j\in [k]$ are invertible  in $\mathcal{R}$.
%    \begin{enumerate}[label=(\alph*)]
    \begin{enumerate}
	\item[1a.] If $|nA| \le |A|^K$ for $n \ll K^2 \log^2 |A| \log (K\log |A|)$, 
    then $\dim_k (A) \ll K \log |A| / \log (K\log |A|)$.
    \item[1b.] Now suppose that  $\dim_k (A) \le K \log |A| / \log (K\log |A|)$.
    Then $|nA| = |A|^{O(K)}$ for all $n = \dim^{O(1)} (A)$.
    \item[2a.]  Further if $\dim_k (A) \le K \dim (A)/ \log \dim (A)$, then $\dim (Q(\Lambda_k)) \ll K \dim (A)$.
    \item[2b.]  If $\dim (Q(\Lambda_k)) \le K \dim (A)$, then $\dim_{k} (A) \ll K \dim (A)/ \log \dim (A)$.
    \end{enumerate}
 \end{theorem}

 Thus, roughly speaking, we have the following equivalencies 
\[
    |nA| \le |A|^K \,, \forall n \le \dim^{O(1)} (A) 
    \quad  \mbox{iff} \quad
    \dim_k (A) \sim_K \frac{\log |A|}{\log \log |A|}
    \quad  \mbox{iff} \quad 
    \dim (Q(\Lambda_k)) \sim_K \dim (A) \,,
\]
    where $k= \dim (A) \log \dim (A)$, 
    and hence $A$ has a very strong additive structure if $\dim_k (A) \sim_K \frac{\log |A|}{\log \log |A|}$, i.e. if $\dim_k (A)$ is rather small. 
    %Our 
    This 
    description of additively rich sets in terms of $\dim_k (A)$ seems to be new.

In our next Section \ref{sec:T_k} and, partially, in Section \ref{sec:subgroups}, we develop this approach and obtain several  %connection 
relations  
between further variants of additive  dimensions and another quantity, which is connected with the higher sumsets (namely, the quantity $\T_k(A)$ see formula \eqref{def:T_k_intr} below or % the definition in 
Section \ref{sec:definitions}). 
These results are applied in the rest of the paper and we formulate just a simple consequence of Theorem \ref{t:T_k_dim} below.

\begin{theorem}\label{t:T_k_dim_intr}
    Let $\Gr$ be an abelian group, $A\subseteq \Gr$ be a set, $\dim (A) := M \log |A|$.
    \begin{enumerate} 
    \item Then there is $m$, $\log |A|/\log M \ll m \le \dim (A)/2$ such that $\T_m (A) \gg |A|^{2m} \exp (-\Omega (\dim (A)))$.
    \item 
    On the other hand, for any $m$ such that $\log |A|/\log M \ll m \le \dim (A)/2$ there exists a set $A_* \subseteq A$ with  $\T_{m} (A_*) \ge 2^{-m} \T_m (A)$ and $\dim (A_*) \le \exp (O(M \log M)) \cdot \log |A|$. 
    \end{enumerate} 
\end{theorem}

\subsection{Applications}

%\noindent
We now describe some applications, which can be obtained via this method.

Multiplicative subgroups in the prime field is a classical theme of number theory see, e.g.,  book \cite{KS}. Basis properties of 
%very small 
subgroups were studied in \cite{Bourgain_coset,K,KS,s_ineq} and in many other papers. 
For example, in \cite[Theorems 2,5]{K} the following result was obtained.

\begin{theorem}
    Let $p$ be a prime number, $\eps \in (0,1)$ be a real number, $\G < \F_p^*$ be a multiplicative subgroup, $|\G| \ge \frac{\log (p/|\G|)}{(\log (\log (p/|\G|)+1))^{1-\eps}}$. 
    Then 
    %there is 
\begin{equation}\label{f:K_introduction}
    n \G = \F_p  \,, 
    \quad \quad \mbox{where} \quad \quad  
    n = O_\eps (\log^{2+\eps} (p/|\G|)) \,.
\end{equation}
    On the other hand, there are infinitely many primes $p$ and $\G < \F_p^*$
    such that $n\G \neq \F_p$, provided\\ 
    $n= O(\log (p/|\G|))$ if $|\G| \ge \log (p/|\G|)$, and\\
    $n= O(\log^{1-\eps} (p/|\G|))$ if $|\G| \ge \log^C (p/|\G|)$ for any constant $C\ge 1$. 
\label{t:K_introduction}
\end{theorem}

Theorem \ref{t:K_introduction} gives an affirmative answer to a question of Heilbronn, see \cite{Heilbronn_subgroups}. 
As Konyagin writes in \cite{K} the conjectured  upper bound  for $n$ in \eqref{f:K_introduction} is, probably, $n = O_\eps (\log^{1+\eps} (p/|\G|))$.

In a natural way, studying multiplicative subgroups $\G$,  the authors of papers \cite{Bourgain_coset,K,KS,s_ineq} applied 
%the authors of these papers
%were interested in obtaining 
upper bounds for exponential sums over such subgroups.
%which belongs to Konyagin \cite{K} 
For example, the upper bound for the Fourier coefficients of $\G$, which allows to obtain Theorem \ref{t:K_introduction} is sharp in a sense but as we have said before the number $n$ can be probably decreased. 
In our approach we do not use this machinery, connected with exponential sums, but rather good lower bounds for some  additive dimensions of $\G$. Let us formulate  our result, see Corollaries \ref{c:G_growth}, \ref{c:p^c} of Section \ref{sec:subgroups}.

\begin{theorem}
    Let $p$ be a prime number and $\G<\F_p^*$ be a multiplicative subgroup.
    \begin{enumerate}
	    \item
    Suppose that $\_phi (|\G|) \log |\G| \ge \log p$ and 
    $|\G| \le (\log p)^C$, where $C\ge 1$  is an absolute constant. 
    Then there is $n = O(\log^2 p/ \log \log p)$ such that $|n\G| \ge p^{\Omega(1/C)}$.
    %for a certain $c=c(C)>0$.  
    \item Further if $|\G|\le \log p$, then for $n = O(\_phi^2 (|\G|) \log |\G|)$ one has 
    $|n\G| \ge \exp(\log |\G| \cdot \Omega( \min\{ \_phi (|\G|), \frac{\log p}{\log \log p}\}))$.
    \end{enumerate} 
\label{t:p^c_intr}
\end{theorem} 

Thus, say, for $|\G| \gg \log p \cdot \frac{\log \log \log p}{\log \log p}$ we obtain $|n\G| \gg p^c$ for a certain absolute constant $c>0$ and 
%$n=o(\log^2 p)$. 
$n = O(\log^2 p/ \log \log p)$. 
Actually, we show that the sequence $n\G$ growths {\it almost  optimally} for small $n$ and moreover, it is possible to estimate the energy 
\begin{equation}\label{def:T_k_intr}
    \T_k^{+} (\G) := |\{ (\gamma_1, \dots, \gamma_k,\gamma'_1, \dots, \gamma'_k) \in \Gamma^{2k} ~:~
    \gamma_1 + \dots + \gamma_k = \gamma'_1 + \dots + \gamma'_k \} |\,,
\end{equation}
see Corollary \ref{c:G_growth}. For example, if $|\G|\sim \log p$, then we obtain for all $n$ 
%\begin{equation}\label{f:growth_bound3}
\begin{equation}\label{f:nG_intr}
    |n\G| = \Omega \left( \left( \frac{|\G|}{n \log^3 |\G|} \right)^n \right) \,, 
    \quad \quad \mbox{and} \quad \quad 
    \T^{+}_n (\G) \le |\G|^n (C_* n \log^3 |\G|)^n \,,
\end{equation}
    where $C_* >0$ is an absolute constant. 
%\end{equation}

\bp 

Any multiplicative subgroup is a set with small product set
and in formula \eqref{f:nG_intr} or in Theorem  \ref{t:p^c_intr} we have 
%considered 
estimated 
some additive characteristics of our subgroup. 
This effect belongs to the wider and well--known sum--product phenomenon see, e.g.,  \cite{TV}. In Section \ref{sec:sum-product} we consider the case 
when our set $A$ belongs to a ring $\mathcal{R}$ 
and we study 
%both 
dimensions $\dim^{+} (A)$, $\dim^{\times} (A)$, which are defined for both ring's operations $+$ and $\times$. In particular, 
%we 
it is possible to 
obtain the following sum--product--type result. 
%$\max\{ \dim^{+} (A), \dim^{\times} (A)\}$ consider 

\begin{theorem}
    Let $A\subset \Z$ be an arbitrary finite set. Then 
\begin{equation}\label{f:sum-prod_dim_lower_intr}
    \max\{ \dim^{+} (A), \dim^{\times} (A)\}
        \gg 
            \log |A| \cdot \frac{\sqrt{\log \log |A|}}{\sqrt{\log \log \log |A|}} \,.
\end{equation}
    On the other hand, there is a set $A\subset \Z$ such that 
\begin{equation}\label{f:sum-prod_dim_intr}
    \max\{ \dim^{+} (A), \dim^{\times} (A)\}
        \ll 
            \log |A| \cdot \log \log |A| \,.
\end{equation}
\label{t:sum-prod_dim_intr}
\end{theorem}

Thus, surprisingly, both dimensions $\dim^{+} (A)$, $\dim^{\times} (A)$ can be rather small. 
Further, using some variations of the  method, we improve a decomposition result from \cite[Corollary 1.3]{Mudgal}. 
The first Theorem  on decompositions of a set  onto two sets with small $\T^{+}_s, \T^\times_s$ was obtained in \cite{BW}.

\begin{theorem}
    Let $A \subset \Z$ be a set and $s$ be an integer parameter,
    %such that 
    \begin{equation}\label{cond:dec_Tk_c_intr}
    s \ll \frac{\log |A|}{\sqrt{\log \log |A| \cdot \log \log \log |A|}} \,.
    \end{equation} 
    Then there exist pairwise disjoint sets $B$ and $C$ such that $A=B\bigsqcup C$ and 
\begin{equation}\label{f:dec_Tk_c_intr}
    \max\{ \T^{+}_s (B), \T^{\times}_s (C) \} \le |A|^{2s-\frac{c_* \sqrt{\log s}}{\sqrt{\log \log s}}} \,,
\end{equation}    
    where $c_*>0$ is an absolute constant.
\label{t:dec_Tk_intr}
\end{theorem}

Basically, we use just one induction in our proof and that is why we have just one logarithm in 
%our 
estimate \eqref{f:dec_Tk_c_intr} in contrast to paper \cite{Mudgal}.  
A similar construction as in Theorem \ref{t:sum-prod_dim_intr} (see \cite[Proposition 1.5]{ZP}) shows that one cannot obtain 
%something 
better saving than $\Omega(\log s/\log \log s)$ in estimate \eqref{f:dec_Tk_c_intr}.

Following the argument of the proof from paper \cite[Theorem 1.1]{Mudgal_Sidon} one can show that Theorem \ref{t:dec_Tk_intr} implies a result on additive/multiplicative Sidon sets in $\Z$ (the definitions can be found in \cite{Mudgal_Sidon}, also, see  preceding   paper \cite{s_Sidon} where the author has to deal with the real case).
%, see the details of the proof in  \cite{Mudgal_Sidon}. 
Recall that a finite set $A\subseteq \Gr$ is a $B^{+}_h [g]$ set if for any $x\in \Gr$ the number solutions to the equation $x=a_1+\dots + a_h$ does not exceed $g$ 
%$r_{hA} (x) \le g$ 
(and similarly for $B^{\times}_h [g]$). 
Estimate \eqref{f:Mudgal_Sidon_2_intr} is a quantitative analogue of the main result from \cite{BC}.

\begin{theorem}
    Let $h$ be a positive integer, $A\subset \Z$ be a finite set, and let $B$ and $C$ be the largest $B^{+}_h [1]$ and $B^{\times}_h [1]$ sets in $A$ respectively. Then 
\[
    \max\{ |B|, |C| \} \gg |A|^{\eta_h/h} \,,
\]
    where $\eta_h \gg (\log h)^{1/2-o(1)}$.
    In particular, for any $A\subset \Z$ 
\begin{equation}\label{f:Mudgal_Sidon_2_intr}
    |hA| + |A^h| \gg_h |A|^{(\log h)^{1/2-o(1)}} \,.
\end{equation} 
\label{t:Mudgal_Sidon_intr}
\end{theorem}

Finally, we have obtained a sum--product--type result  of another sort. In additive combinatorics we believe that higher sumsets have rich additive structure and this is a classical theme of research. On the other hand, thanks to the sum--product phenomenon it means that such sets must have rather poor multiplicative structure. It turns out that this heuristic can be expressed in terms of some dimensions. For example,  we show that for any positive integer $k$ one has 
\[ 
%\begin{equation}\label{f:D_dim_2}
    \dim^\times_k (D) \ge \exp (\Omega(\log |A|/\log K)) \,,
%\end{equation}
\]
where $A$ is an arbitrary  set with $|A+A|\le K|A|$ and 
$D:=A-A$.
This result on dimension of the difference sets 
is interesting in its own right, especially, due to our crucial inclusion 
\begin{equation}\label{f:D_dim_intr}
    \{ 1,2, \dots, n\} \subset \frac{\{ 1,2, \dots, n\}}{\{1,2, \dots, n\}} \subseteq \frac{D}{D} \,, \quad \quad \mbox{ where }
    \quad \quad 
    n = \exp (\Omega(\log |A|/\log K)) \,.
\end{equation}
%and $D=A-A$ for an arbitrary  set $A$ with $|A+A|\le K|A|$. 
Estimate \eqref{f:D_dim_intr} is exponentially better than the lower bound for the length of the largest arithmetic progression in $A\pm A$ from  \cite{CRS}.

\section{Definitions and preliminaries}
\label{sec:definitions}

By $\Gr$ we denote an abelian group.
% with the identity $1$.
Sometimes we underline the group operation writing $+$ or $\times$ in  the considered quantities (as energies, the representation  function, dimensions  and so on).
%, see below). 
Sometimes it is useful to consider  an abelian ring $\mathcal{R}$ instead of our group $\Gr$. Of course in this case one can apply two ring's operations simultaneously. 
We use the same capital letter to denote  set $A\subseteq \Gr$ and   its characteristic function $A: \Gr \to \{0,1 \}$. 
Given two sets $A,B\subset \Gr$, define  
%the \textit{product set} (
the {\it sumset} 
%in the abelian case) 
of $A$ and $B$ as 
$$A+B:=\{a+b ~:~ a\in{A},\,b\in{B}\}\,.$$
In a similar way we define the {\it difference sets} and {\it higher sumsets}, e.g., $2A-A$ is $A+A-A$.
Given a positive integer $k$, we 
%define 
put 
$$\Sigma_k (A) := \left\{ \sum_{s\in S} s ~:~ S\subseteq A,\, |S| \le k \right\} \,.$$ 
Clearly, $kA \subseteq \Sigma_k (A)$, and if $0\in A$, then $\Sigma_k (A) = kA$. Also, trivially, $|\Sigma_k (A)| \le k |kA|$.
For an abelian group $\Gr$
the Pl\"unnecke--Ruzsa inequality (see, e.g., \cite{TV}) holds stating
\begin{equation}\label{f:Pl-R} 
|nA-mA| \le \left( \frac{|A+A|}{|A|} \right)^{n+m} \cdot |A| \,,
\end{equation} 
where $n,m$ are any positive integers. 
Further if $|A+B|\le K|A|$ for some sets $A,B \subseteq \Gr$, then for any $n$ one has 
\begin{equation}\label{f:Pl-R+} 
    |nB| \le K^n |A| \,.
\end{equation}
If $A\subseteq \mathcal{R}$ and $\lambda \in \mathcal{R}$, then we write $\lambda \cdot A := \{\lambda a ~:~ a\in A\}$. 
We  use representation function notations like  $r_{A+B} (x)$ or $r_{A-B} (x)$ and so on, which counts the number of ways $x \in \Gr$ can be expressed as a sum $a+b$ or  $a-b$ with $a\in A$, $b\in B$, respectively. 
For example, $|A| = r_{A-A} (0)$.
%Also, we write $A^{(j)} (x) = r_{j$
For any two sets $A,B \subseteq \Gr$ the {\it additive energy} of $A$ and $B$ is defined by
$$
\E (A,B) = \E^{+} (A,B) = |\{ (a_1,a_2,b_1,b_2) \in A\times A \times B \times B ~:~ a_1 - b^{}_1 = a_2 - b^{}_2 \}| \,.
$$
If $A=B$, then  we simply write $\E^{} (A)$ for $\E^{} (A,A)$.
More generally, for sets (real functions) $A_1,\dots, A_{2k}$ ($f_1,\dots, f_{2k}$) belonging to an arbitrary (noncommutative) group $\Gr$ and 
%integer 
$k\ge 2$ define the energy 
$\T_{k} (A_1,\dots, A_{2k})$ as 
%как число решений уравнения 
\[
	\T_{k} (A_1,\dots, A_{2k}) 
	=
\]
\begin{equation}\label{def:T_k}
	=
	 |\{ (a_1, \dots, a_{2k}) \in A_1 \times \dots \times A_{2k} ~:~ a_1 a^{-1}_2  \dots a_{k-1} a^{-1}_k = a_{k+1} a^{-1}_{k+2}  \dots a_{2k-1} a^{-1}_{2k}  \}| \,,
\end{equation}
and 
\[
	\T_{k} (f_1,\dots, f_{2k}) 
	=
	\sum_{a_1 a^{-1}_2  \dots  a_{k-1} a^{-1}_k = a_{k+1} a^{-1}_{k+2} \dots a_{2k-1} a^{-1}_{2k}} f_1(a_1) \dots f_{2k} (a_{2k}) \,.
\]
One has (see, e.g., \cite[Section 2.3]{TV}) 
\begin{equation}\label{f:T_k_prod}
    \T_{k} (f_1,\dots, f_{2k}) \le \prod_{j=1}^{2k} \T^{1/2k}_k (f_j) \,.
\end{equation}
In particular, $\T^{1/2k}_k (f)$ defines a norm. 
For any function $f:\Gr \to \mathbb{C}$ and $\rho \in \FF{\Gr}$ define 
%the matrix $\FF{f} (\rho)$, which is called
the Fourier transform of $f$ at $\rho$ by the formula 
\begin{equation}\label{f:Fourier_representations}
\FF{f} (\rho) = \sum_{g\in \Gr} f(g) \rho (g) \,.
\end{equation}

Given a set $A = \{a_1,\dots, a_n\} \subseteq \Gr$, recall that the {\it combinatorial cube} $Q(A)$ is $\{0,1\}\cdot a_1 + \dots + \{0,1\} \cdot a_n$. If $|Q(A)| = 2^n$, then the cube is called {\it proper}. Similarly, if $P_1,\dots, P_d \subseteq \Gr$ are some arithmetic progressions, then the set  $S:= P_1+\dots+P_d$ is called {\it generalized arithmetic progression of dimension} $d$ and if $|S|= \prod_{j=1}^d |P_j|$, then $S$ is {\it proper}.

The signs $\ll$ and $\gg$ are the usual Vinogradov symbols.
When the constants in the signs  depend on a parameter $M$, we write $\ll_M$ and $\gg_M$. 
If $a\ll_M b$ and $b\ll_M a$, then we write $a\sim_M b$. 
All logarithms are to base $2$.
By $\F_p$ denote $\F_p = \Z/p\Z$ for a prime $p$. Let $\F^*_p = \F_p \setminus \{0\}$. 
If we have a set $A$, then we will write $a \lesssim b$ or $b \gtrsim a$ if $a = O(b \cdot \log^c |A|)$, $c>0$.
Let us denote by $[n]$ the set $\{1,2,\dots, n\}$.

\bp 

Given a positive integer $k$ 
%and a set $S = \{ s_1,\dots, s_n\}$ 
one can define {\it $k$--dissociated} set $\Lambda = \{\la_1,\dots,\la_n\}$ if any equality of the form
\begin{equation}\label{def:diss_eq}
    \sum_{j=1}^n \eps_j \la_j = 0 \,, 
    \quad \quad \mbox{where}  \quad \quad 
    |\eps_j|\le k
\end{equation}
implies that all $\eps_j$ are equal to zero.
In contrary, if there is a tuple  $(\eps_1, \dots, \eps_n)$ such that \eqref{def:diss_eq} holds and such that  not all $|\eps_j| \le k$ are equal to zero, then we say that $(\la_1,\dots, \la_n)$ forms an {\it additive $n$--tuple}.
Thus a $k$--dissociated set $\Lambda$ has no additive  $|\Lambda|$--tuples. 
If $A\subseteq \Gr$ is a set, then we write $\dim_k (A)$ for the size of the largest $k$--dissociated subset of $A$. 
In particular, $\dim (A) = \dim_1 (A)$. 
Clearly, 
\begin{equation}\label{def:another_def}
    \dim_k (A) = \min \{ t ~:~ \forall a=(a_1,\dots, a_{t+1}) \in \ov{A}^{t+1},\, ~ \exists \eps = (\eps_1,\dots,\eps_{t+1}) \in [-k,k]^{t+1} \mbox{ s.t. } \langle \eps, a \rangle = 0 \} \,.
\end{equation}
The set $\ov{A}^{t+1}$ from formula \eqref{def:another_def} 
%we have assumed that 
is the set of 
all tuples $a_1,\dots, a_{t+1}$
%are different. 
with  different elements. 
Notice that $\dim_k (A) = \dim_k (A\cup \{0\})$ and  $\dim_k (A) \gg \log_k |A|$. 
Clearly, $\dim_{k_1} (A) \le \dim_{k_2} (A)$ if $k_1\ge k_2$ but also almost obviously, that all these dimensions are weakly equivalent for different $k_j$ see, e.g.,  formula 
%\eqref{f:d_dim} 
%\eqref{f:dim_compare}
\eqref{f:dim_compare+} of Lemma \ref{l:dim_compare+} 
below. 
%Also, notice that $\dim_k (A) \gg \log_k |A|$. 
For any $k$ the dimension is monotone and subadditive, that is,  $\dim_k (B) \le \dim_k (A)$ for any $B\subseteq A$ and for arbitrary $B_1,\dots, B_t \subseteq \Gr$ one has 
$$
    \dim_k (\bigcup_{j=1}^t B_j) \le \sum_{j=1}^t \dim_k (B_j) \,. 
$$

Similarly, given a set $S\subseteq \Gr$ and a positive integer $k$ define
    $\Span_k (S) = \{ \sum_{j=1}^n \eps_j s_j ~:~ |\eps_j| \le k \}$
    and let 
$$ 
d^*_k (A) := \min \{ |S| ~:~ A \subseteq \Span_k (S)\} 
\quad \quad 
\mbox{and}
\quad \quad 
d_k (A) := \min \{ |S| ~:~ S \subseteq A \subseteq \Span_k (S)\} \,.
$$
We write $d^*(A)$ for $d^*_1 (A)$ and $d (A)$ for $d_1 (A)$. 
Clearly,
$d^* (A) \le d (A) \le \dim (A)$, $d^*_k (A) \le d_k (A)$ for any $k$ and if $A$ is a dissociated set, then $d(A) = \dim (A) = |A|$ (in contrary, there is a dissociated set $A$ such that $d^*(A) \sim  \dim (A)/\log \dim (A)$, see Example \ref{exm:cube_with_large_dim} below). 
%$d^*_k (A) \le \dim_k (A)$. 
Further  $d^*_k (A)$ is monotone 
(but $d_k(A)$ is not, see \cite[Example 8.1]{ss_dim})
%or Example \ref{exm:cube_with_large_dim} below) 
and 
both 
$d^*_k (A)$, $d_k (A)$ are 
subadditive, as well as the dimension $\dim_k  (A)$.
Again, notice that $d^*_k (A) = d^*_k (A\cup \{0\})$, 
$d_k (A) = d_k (A\cup \{0\})$ 
and also $d^*_k (A) \le k \dim_k (A)$ if $A$ belongs to a ring $\mathcal{R}$ such that all elements from $[k]$ are invertible.  
Clearly, $d_{k_1} (A) \le d_{k_2} (A)$  and  $d^*_{k_1} (A) \le d^*_{k_2} (A)$ if $k_1\ge k_2$.
Unlike $\dim_k (A)$  the dimensions 
$d_k (A)$, 
$d^*_k (A)$ enjoy   the following properties 
\begin{equation}\label{f:d^*_k_addition}
    d^*_{kt} (B_1+\dots+B_t) \le \sum_{j=1}^t d^*_k (B_j)  
\end{equation}
for arbitrary sets  $B_j \subseteq \Gr$, $j\in [t]$
and
\begin{equation}\label{f:d_k_addition}
    d^*_k (B_1+\dots+B_t) \le \sum_{j=1}^t d_k (B_j)  
\end{equation}
for  any 
%arbitrary 
disjoint sets $B_j \subseteq \Gr$, $j\in [t]$. 
%Namely, 
%In particular, 
%\begin{equation}\label{f:dim_compare}
%    \dim_l (A) \ll k\dim_{k} (A) \cdot \log_{l+1} (k\dim_k (A)) \,.
%\end{equation}
%The proof is straightforward and repeats the calculations in formulae \eqref{fi:pol_growth_1}, \eqref{fi:pol_growth_2} of Proposition \ref{p:pol_growth}. 
For other variants of additive dimensions of a set consult \cite{ss_dim} and \cite{CH}.  We show in the next section that all such dimensions differ by some logarithmic factors, basically, this result is contained in papers \cite{CH}, \cite{LY}. 
%and hence equivalent upto. 

\bp 

To obtain our  applications in Sections \ref{sec:subgroups}, \ref{sec:sum-product} we need some results.
First of all, we recall the well--known Theorem  of Rudin \cite{Rudin_book} on dissociated sets.

\begin{theorem}
    Let $\Lambda \subseteq \Gr$ be a dissociated set.
    Then for any positive integer $k$ one has 
\[
    \T_k (\Lambda) \le (Ck)^k |\Lambda|^k \,, 
\]
    where $C>0$ is an absolute constant. 
\label{t:Rudin}
\end{theorem}

Now let $A\subseteq \Gr$ be a set. Put
\begin{equation}\label{def:beta}
    \beta (A) := \inf_{X,Y \neq \emptyset} \frac{|A+X+Y|}{\sqrt{|X||Y|}} \ge 1 \,.
\end{equation}
The quantity  $\beta (A)$ is discussed in detail  in \cite{MRSZ} and in \cite{BMRSZ}.  Basically, $\beta (A)$ measures the additive structure of our set $A$.
Notice that by \eqref{def:beta} and induction we have for any integer  $k\ge 1$
\begin{equation}\label{def:2^k_beta}
    |(2^k-1)A| \ge \beta^{k-1} (A) \cdot |A| \,.
\end{equation} 
Thus, applying Proposition \ref{p:pol_growth} below (e.g., consult formula \eqref{fi:pol_growth_2}), we see that $\log \beta(A) \ll \dim (A)$ but more precisely by \cite[Statement 3.2]{MRSZ} one has 
\begin{equation}\label{f:beta_dim}
    \beta(A) \le 2^{\dim (A)} \,.
\end{equation}
%The author thinks 
It seems like 
that simple estimate  \eqref{f:beta_dim} is the only relation between $\beta (A)$ and $\dim (A)$.
Indeed, consider the following instructive example: 
%, e.g., 
take any set $A\subseteq \Z$ and  then (see \cite[Statement 3.2]{MRSZ}) one has  $\beta (A) \le 2$ but  $\dim (A)$ can be arbitrary large.

We need the beautiful result on the connection of $\T^{+}_k (A)$ and the quantity $\beta(A)$ for 
%integer sets 
$A \subset \Z$, see \cite[Theorem 1.3]{ZP}. 

\begin{theorem}
    Let $A\subset \Z$ be an arbitrary finite set and $\eps \in (0,1)$, $k\ge 2$ be parameters. Then 
\[
    \T^{+}_k (A) \le 10^k \beta^{k/\eps}(A) |A|^{k+2\eps k \log k} \,.
\]
\label{t:ZP}
\end{theorem}

We also use a classical result of Bukh \cite{Bukh} that provides an upper bound for a sum of dilates of a set in terms of the additive doubling of that set.

\begin{theorem}\label{t:bukh}
	Let  $A$ be a finite subset of  an abelian group such that $|A+A|\le K|A|$. Then  for any  $\lambda_{i} \in \mathbb{Z} \backslash\{0\}$ we have
	$$
	\left|\lambda_{1} \cdot A+\cdots+\lambda_{k} \cdot A\right| \leqslant K^{O\left(\Sigma_{i=1}^k \log \left(1+\left|\lambda_{i}\right|\right)\right)}|A| \,.
	$$
\end{theorem}

Also, we need a relaxation of the well--known sunflower lemma of Erd\H{o}s  and Rao \cite{ER_sunflower}, which is due to F\"uredi \cite{Furedi}. 
Actually in our regime there is almost no difference between these two results, as well as between a modern relaxation of the sunflower lemma, see  \cite{Rao_sunflower}. 
Recall that given a family of sets $A_1,\dots, A_r$, their {\it common part} is the set 
\[
    X := \bigcup_{i\neq j} (A_i \cap A_j) \,. 
\]
Note that, if $|X| < \min_i |A_i|$, then all the sets $A_1\setminus X, \dots, A_r\setminus X$ are nonempty and mutually disjoint and thus it is a relaxation of the notion of the classical sunflower. 
Now given two positive integers $k$ and $r$, let  $f(k,r)$ be the smallest number $n$ such that any collection $\mathcal{C}$ of more than $n$ sets, every $C\in \mathcal{C}$ has size 
%at most 
$k$, then $\mathcal{C}$ contains $r$ members with the common part less than $r$.  
%a sunflower of $r$ sets. 
We have the following result, see \cite{Furedi}.

\begin{lemma}\label{l:sunflower}
    Let $k$ and $r$ be positive integers. 
    Then $f(k,r) \le (r-1)^k$. 
\end{lemma}

%The modern bounds (we do not need them in our arguments) and the history of the question can be found in \cite{Rao_sunflower} and in \cite{Kostochka_sunflower}.

%\bigskip 

Finally, we discuss a simple connection between additive dimension and Diophantine approximations, which is, actually, well--known see, e.g., \cite{Bourgain_AP3}, \cite{Bourgain_AP3_new}, \cite{c1}, \cite{Sanders_1}, \cite{Sanders_Bog} etc.

Let $N$ be a positive integer. 
%prime number.
Given a positive real number  $s$  and a set $A\subseteq \Z/N\Z$ we define
\begin{equation}\label{def:d_s(A)}
    \mathcal{D}_{s,N} (A) := \min_{q \in [N-1]}\, \| qA\|_{s,N}^s  
    = \min_{q \in [N-1]}\, \sum_{a\in A} \left\| \frac{q a}{N} \right\|^s \,.
\end{equation}
Similarly, if $A\subset \R$ and $N$ is a positive integer, then we put
%prime $N$ or where 
%in the case 
\begin{equation}\label{def:d_s(A)_R}
    \mathcal{D}_{s,N} (A) := \min_{q \in [N-1]}\, \| qA\|_{s,N}^s  
    = \min_{q \in [N-1]}\, \sum_{a\in A} \left\| aq \right\|^s \,.
\end{equation}
%For $A\subseteq \Z/N\Z$ we write just $\| qA\|_{s}^s$ instead of  $\| qA\|_{s,N}^s$  and $d_{s} (A)$ instead of $d_{s,N} (A)$ in this case.
Finally, notice that if $A\subset \Z$ and $A_N := A \pmod N$, then always $\dim(A) \ge \dim (A_N)$.

%\bigskip 

%%In our next lemma we describe a simple connection between $\mathcal{D}_{s,N} (A)$ and $\dim(A)$.
%for a set with small multiplicative doubling. 

\begin{lemma}
    Let $N$ 
    %be a prime, 
    be a positive integer, 
    $A\subset \R$ or $A\subseteq \Z/N\Z$.
    %%be a set such that all elements of $A$ are coprime with $N$ {\bf ???}.
%    Also, $A'$ the set of all elements of $A$ which are coprime to $N$ and $|A'|=\zeta |A|$, $\zeta \in (0,1]$.
    Put $d=\dim (A)$ and 
    suppose that for a certain $s>0$ one has 
    $\mathcal{D}_{s,N} (A) \ge 
    %\_phi(N)^s 
    |A|/T$, $T\ge 1$.
    Then 
\begin{equation}\label{f:AA_d_s}
    d \ge \frac{s \log (N-1)}{\log (dT)} \,.
\end{equation}
\label{l:AA_d_s}
\end{lemma}
\begin{proof}
    We consider the case $A\subseteq \Z/N\Z$ because if $A\subset \R$, then the argument is the same. 
    Let $\Lambda = \{\la_1,\dots, \la_d\} \subseteq A$ be a maximal dissociated subset of $A$, $d=\dim (A)$. 
%    Thus any element of the set $AA$ can be expressed as $\sum_{j=1}^{d} \eps_j \lambda^*_j$, where $\eps_j \in \{0,\pm 1\}$. 
    Using the Dirichlet Theorem, we find 
    %$q \in (\Z/N\Z)^*$ 
    $q\in (\Z/N\Z) \setminus \{0\}$
    such that 
    $\| q\lambda_j/N\| \le (N-1)^{-d^{-1}}$ for all $j\in [d]$.
    Hence 
\[
    \frac{|A|}{T} \le \mathcal{D}_{s,N} (A) \le \sum_{a\in A} \left\| \frac{qa}{N} \right\|^s \le |A| d %\_phi(N)^{-sd^{-1}}
    (N-1)^{-sd^{-1}}
\]
     as required.  
%    and $\Lambda'$ be a maximal dissociated subset of $A'$.
\begin{comment} 
    By subadditivity and our assumption, we have $d_* := \dim (AA) \le D \dim (A) = Dd$ and hence any element of the set $AA$ can be expressed as $\sum_{j=1}^{d_*} \eps_j \lambda^*_j$, where $\eps_j \in \{0,\pm 1\}$ and $\lambda^*_j$, $j\in [d_*]$ form a maximal dissociated subset of $AA$.
    Using the Dirichlet Theorem, we find $q \in (\Z/N\Z)^*$ such that 
    $\| q\lambda^*_j/N\| \le \_phi(N)^{-d^{-1}_*}$.
    We have 
\[
    \sum_{z\in A} \| qz\Lambda\|_s^s = \sum_{z\in A}\, \sum_{\lambda\in \Lambda} \| qz\lambda/N\|^s 
    =
    \sum_{\lambda \in \Lambda}\, \sum_{z\in A} \|
    qz\lambda/N\|^s \ge |\Lambda| d_s (A) \,.
\]
%    where we have denoted by $A'$ the set of all elements of $A$ which are coprime with $N$. 
    %the summation is taken over $z\in A'$ coprime with $N$.
%    Since $|A\Lambda^{-1}|\le |AA^{-1}|\le D^2 |A|$ by the Pl\"unnecke inequality, we see that there is $z\in A\Lambda^{-1}$ such that $\| z\Lambda\|_s^s \ge \frac{|\Lambda|}{D^2 |A|} \| A\|_s^s$.
    Hence we see that there is $z\in A$ such that $\| qz\Lambda\|_s^s \ge \frac{|\Lambda|}{|A|} d_s (A)$.
    But $z\Lambda \subseteq AA$ whence for any $\la \in \Lambda$ the element $z\lambda$ can be expressed as $\sum_{j=1}^{d_*} \eps_j \lambda^*_j$ and thus  
\[
%    \frac{d \_phi(N)^s}{T} 
    \frac{d}{T} 
    \le   
    \frac{d}{|A|} d_s (A) \le \sum_{\la \in \Lambda} \| qz\lambda/N \|_s^s \le dd_* \_phi(N)^{-sd^{-1}_*}  
    \le d^2 D \_phi(N)^{-sd^{-1}_*}  
\]
\end{comment} 
% as required.  
%    This completes the proof. 
$\hfill\Box$   
\end{proof}

\begin{comment} 
\begin{proposition}
    Let $x$, $y$ be positive integers and  $I = x + [y]$.
    Then 
    $$
    \dim^\times (I) \gg \sqrt{y} \,. 
    $$
\end{proposition}
\begin{proof} 
    For $x\le y$ the result follows from the prime number theorem.
    Now let $y^l < x \le  y^{l+1}$, where $l$ be an integer.   %parameter.
    Using the prime number theorem again, we can find the number of primes in $I$ and hence
\[
    \dim^\times (I) \ge \frac{x+y}{\log (x+y)} - \frac{x}{\log x}
    = \frac{y}{\log (x+y)} - \frac{x \log (1+y/x)}{\log x \cdot \log (x+y)}
        \gg \frac{y}{l \log y} - \frac{y}{l^2 \log y}
        \gg \frac{y}{l \log y} \,.
\]
% as required.  
%    This completes the proof. 
$\hfill\Box$   
\end{proof}
\end{comment}

\section{General results on additive dimensions}
\label{sec:general}

We start this section  with a discussion concerning relations between different types of dimensions. 
As we have said before all dimensions $\dim_k (A)$, $d^*_k (A)$, $d_k(A)$ differ by some logarithmic factors but we now obtain a  more concrete result. 
%better result. 
%{\bf Reversibility!!!!}

\begin{lemma}
    Let $k,l$ be positive integers, $\mathcal{R}$ be a ring such that all numbers $j\in [k]$ are invertible   and $A\subseteq \Gr$ be a set. 
    Then 
\begin{equation}\label{f:dim_compare+}
    \dim_l (A) \ll \dim_k (A) \log_{l+1} (k\dim_k (A)) \,.
\end{equation}
    Similarly, 
\begin{equation}\label{f:dim_compare+2}
    d_l (A) \le \dim (A) \ll \dim_k (A) \log_{} (k\dim_k (A)) \,.
\end{equation}
\label{l:dim_compare+}
\end{lemma} 
\begin{proof} 
    Let $\Lambda_l$ and $\Lambda_k$ be some maximal $l$ and $k$--dissociated subsets of $A$, correspondingly. 
    Thus $|\Lambda_l|= \dim_l (A)$ and $|\Lambda_k|= \dim_k (A)$. 
    By $l$--dissociativity all sums $\sum_{\la \in \Lambda_l} n_\la \la$, where $n_\la \in [0,1,\dots,l]$ are distinct and hence this set of sums $S$ has size $(l+1)^{\dim_l (A)}$.
    On the other hand, for any $\la \in \Lambda_l$ there is $j\in [k]$ such that $j\la \in \Span_k (\Lambda_k)$.
    Splitting all elements of $\Lambda_l$ onto these $k$ sets, we see that 
    %It follows that 
    any element of the set $S$ belongs to $\sum_{j=1}^k j^{-1} \Span_{kl \dim_l (A)} (\Lambda_k)$. 
    Hence we have 
\begin{equation}\label{tmp:dim_compare+}
    (l+1)^{\dim_l (A)} \le \sum_{j=1}^k |j^{-1} \Span_{kl \dim_l (A)} (\Lambda_k)| 
    \le k (2kl \dim_l (A)+1)^{\dim_k (A)}
\end{equation}
    and thus 
\[
    \dim_l (A) \ll \dim_k (A) \log_{l+1} (k \dim_l (A)) %\,.
\]
    as required. 
%
%
%    Now let $A \subseteq \Span_l (S)$ and $|S| = d_l (A)$. 
    To obtain \eqref{f:dim_compare+2} we use inequality \eqref{f:dim_compare+} and see that
\begin{equation}\label{tmp:09.05_1}
    d_l (A) \le d (A) \le \dim (A) \ll \dim_k (A) \log (k \dim_k (A)) \,.
\end{equation}
    This completes the proof. 
$\hfill\Box$    
\end{proof}

%\bigskip

\begin{exm}
\label{exm:1_2}
$1)~$ Let $A=[n]$. Then $\dim^{+}_k (A) \sim \log_k n$ because the set $\{ 1, (k+1), \dots, (k+1)^s \} \subset [n]$, $s= [\log_{k+1} n]$ is $k$--dissociated.
Also, as in formula \eqref{tmp:dim_compare+}, we have  $(k+1)^{\dim^{+}_k (A)} \le 2k\dim^{+}_k (A) n \le 2k n^2$ and thus we obtain $\dim^{+}_k (A) \le 2\log_{k+1} n + O(1)$.
It is easy to see that for $k\ge n$ one has $\dim^{+}_k (A) = 1$ 
and it shows in particular, that inequality \eqref{f:dim_compare+} is tight (take $l=1$ and $k=n$). 
\\
$2)~$ Again let $A=[n]$ but now let us consider $\dim^{\times}_k (A)$.
Since $A$ contains the primes, it follows that for any $k$ one has $\dim^{\times}_k (A) \gg n/\log n$.
On 
%upper bounds 
asymptotic $\dim^\times (A) \sim \pi(n)$ consult  \cite{Erdos_mult}. \end{exm}

%\bp 

Let us consider the 
%following 
second 
important

\begin{exm} 
\label{exm:cube_with_large_dim}
    Let $Q=Q(a_1,\dots,a_d)$ be a proper combinatorial cube, $|Q|=2^d$.
    If the set $\{a_1,\dots,a_d\}$ is sufficiently dissociated, then one can show (or see \cite[Lemma 10]{s_cubes})
    %\cite[Lemma 10]{s_cubes}
    that $\T_n (Q) \ll |Q|^{2n-C\log n}$, where $C>0$ is an absolute constant. 
    On the other hand, by formula  \eqref{f:T_k_dim_2} of Theorem \ref{t:T_k_dim} 
    below 
    one has $\dim (Q) \gg \log |Q| \cdot \log \log |Q| = d \log d > d$ and hence $\dim (Q) > d$
%    It is connected, in particular, with the fact that we allow to lose logarithmic factors in our 
    (actually, one has $\dim (Q) \sim \log |Q| \cdot \log \log |Q|$ thanks to resolution of the {\bf coin weighing problem}, see the references in  \cite{CH} and in \cite{LY} or just our Lemma \ref{l:LY}). %%The same example shows (take $A=Q$, $B=-Q$) that we need in $\log \log \dim (A+B)$
    %some logarithms 
    %%in estimate \eqref{f:dim_sum_2}. 
    Notice also, that, clearly, $d(Q) \le d$ and hence the dimensions $d(A)$ and $\dim (A)$ can differ by a logarithm. 
\end{exm}

%\bp 

Finally, as in the proof of Lemma \ref{l:dim_compare+}, see formulae in 
\eqref{tmp:dim_compare+}
%\eqref{f:dim_sum_2} or see the calculations in formulae \eqref{fi:pol_growth_1}, \eqref{fi:pol_growth_2} of Proposition \ref{p:pol_growth}
(or consult  \cite[Section 8]{ss_dim}) for any positive integers $l$, $k$  and supposing that all numbers in $[\max\{k,l\}]$ are invertible in $\mathcal{R}$, we get  
\begin{equation}\label{f:d_dim}
    d^*_l (A)/l \le 
    \dim_l (A) \ll d^*_k (A) \cdot \log_{l+1} (kd^*_k (A))
    \le 
    k \dim_k (A) \log_{l+1} (k^2 \dim_k (A)) \,.
\end{equation}
In particular (see formulae in \eqref{f:dim_compare+2}, \eqref{tmp:09.05_1}), 
\begin{equation}\label{f:d_dim+}
    d_l (A) \le \dim (A) \ll d^*_k (A) \cdot \log_{} (kd^*_k (A)) \,.
\end{equation}
Thus dimensions $\dim (A), d^*(A), d(A)$ differ by some logarithmic factors and this is tight in general, see Example \ref{exm:cube_with_large_dim}. 
We give another proof of  bound \eqref{f:d_dim}, which allows us to obtain an upper bound for $\dim_k (A)$ in the spirit of \cite[Lemma 7.2]{ss_dim} but here we had  to deal with the dimension $d(A)$ not $\dim (A)$.

\begin{lemma}\label{l:|A|-t}
    Let $\Gr$ be an abelian group, $A\subseteq \Gr$ be a set and $l,k$ be positive integers. 
    Then $\dim_l (A) \ll d^*_k (A) \cdot \log_{l+1} (kd^*_k (A))$.\\
    Also, let $t,s$ be positive integers.
    If $A$ contains $\exp(\Omega(s \log (k(s+t))))$ additive $s$--tuples, then $\dim_k (A) \le |A|-t$. 
\end{lemma}
\begin{proof} 
    Let $d= \dim_l (A)$ and $\Lambda = \{\la_1, \dots, \la_{d}\}$ be a maximal $l$--dissociated subset of $A$ and $S = \{ s_1,\dots, s_{|S|}\}$, $|S| = d^*_k (A)$ be a set such that $A\subseteq \Span_k (S)$.
    It follows that $\la_j = \sum_{i=1}^{|S|} \eps_{ij} s_i$, where $|\eps_{ij}| \le k$. 
    Consider the matrix $\mathcal{E} = (\eps_{ij})$, $i\in |S|$ and $j\in [d]$. 
    Then for any $\vec{\omega} \in [-l,l]^d$ one has $\| \mathcal{E} \vec{\omega}\|_\infty \le lkd$. 
    Considering $(l+1)^d$ vectors with elements from $[0,1,\dots, l]$, we see that if $(l+1)^d > (2lkd+1)^{|S|}$, then there are distinct  $\vec{\omega}', \vec{\omega}'' \in [0,1,\dots, l]^d$ such that $\mathcal{E} \vec{\omega}_1 = \mathcal{E} \vec{\omega}_2$ and hence
    $\mathcal{E} (\vec{\omega}'-\vec{\omega}'') = 0$. 
    We see that $\vec{\omega} = (\omega_1,\dots, \omega_d) := \vec{\omega}'-\vec{\omega}'' \in [-l,l]^d$ and $\sum_{j=1}^d \omega_j \la_j = 0$, contradicting $l$--dissociativity of $\Lambda$. 
    Hence $(l+1)^d \le (2lkd+1)^{|S|}$ and the result follows.

    Let us obtain the second part of Lemma \ref{l:|A|-t}. Let $\Lambda$ be a maximal $k$--dissociated subset of $A = \{a_1,\dots, a_{|A|}\}$ and suppose that $d:=|\Lambda| > |A| - t$. 
    Having an additive $s$--tuple $\sum_{j=1}^d \eps_j a_j = 0$, $|\eps_j| \le k$ with elements belonging to the set $A$, we %associate 
    consider 
    a set of $a_j \in \Lambda$ from this tuple. 
    Clearly, each such a set is counted with multiplicity at most $(2k+1)^s$ and we fix just one set from this ensemble. 
    Thus we have obtained some (different) sets $S_i \subseteq \Lambda$ of size at most $s$ and by Lemma \ref{l:sunflower} we can find the relaxation of a sunflower $S$, $|S| = r$,
    provided there are at least $(2k+1)^s f(s,r)$ additive $s$--tuples (splitting our family of sets one can suppose that the size of sets is exactly $s$). 
    %with the common set $S$, 
    Adding at most $t$ elements to $S$ and 
    %using the previous argument, 
    using the argument of the first part of the lemma, 
    we see that the parameter $r$ must be at most $O((s+t) \log_{k+1} (s+t))$ because otherwise we have a contradiction  with $k$--dissociativity of $\Lambda$. 
    Applying Lemma \ref{l:sunflower}, we obtain the result.  
    %$(r+t) \log_{k+1} (r+t) \ll s$.
    %(2k+1)^s 
 %   Now putting $r=t$, we obtain the result.
%    as required. 
    This completes the proof. 
$\hfill\Box$    
\end{proof}

\bp 

Now we are ready to begin to study the connection between our dimensions and higher sumsets. 
Notice that for any $A,Z \subseteq \Gr$ 
%one has 
the following holds 
\begin{equation}\label{f:A+B_dim}
    |A+Z| \gg |Z| \cdot \frac{\dim (A)}{\log |Z|} \,.
\end{equation}
This bound is contained in \cite[Theorem 4.2]{ss_dim}, where the symmetric case was considered or just see the proof of Lemma \ref{l:dL} below, where we propose slightly different argument.
In particular, for any integer $n\ge 1$ one has 
\begin{equation}\label{f:nA_dim}
    |nA| \gg |A| \cdot \prod_{j=1}^{n-1} \frac{\dim (A)}{\log |jA|}  \,.
\end{equation}
%
%
%
%\begin{comment} 
For the sake of the completeness we give the proof of \eqref{f:A+B_dim} 
 and improve estimate \eqref{f:nA_dim}. 
%(and, actually, a slightly more general fact). 

\begin{lemma}
    Let $A,Z \subseteq \Gr$, $k,m$ be positive integers.
    Suppose that $|m ([k] \cdot A)+Z|\le K|Z|$, $m \le \dim_k (A)/2$, and  $m\ll \log |Z| \le \dim_k (A)$. 
    Then 
\begin{equation}\label{f:dL}
    k \dim_k (A) \ll K^{1/m} \log |Z| \,.
\end{equation}
    In particular, there is an absolute constant $C>0$ such that for any $m$ with 
    %$m \le \dim(A)/2$  
    $m\le C^{-1} \log |A| \le \dim_k (A)/4$
    one has 
\begin{equation}\label{f:dL_nA}
    |m ([k] \cdot A)| \ge |A| \cdot \left( \frac{k \dim_k (A)}{C \log |A|} \right)^{m-1} \,. 
\end{equation}
    Now if $\log |A| \ll m\le \dim_k (A)/4$, then  
\begin{equation}\label{f:dL_nA+}
    |m ([k] \cdot A)| \ge 
    % |A| \cdot 
    \left( \frac{k \dim_k (A)}{4m} \right)^{m-1} \,. 
\end{equation}
%    The same is true if one replaces $\dim(A)$ to $k\cdot \dim_k (A)$ in all formulae above. 
\label{l:dL}
\end{lemma}
\begin{proof} 
    It is sufficient to obtain \eqref{f:dL} because 
    \eqref{f:dL_nA} follows if one takes  $m\to m-1$ and  $Z=A$ in \eqref{f:dL} (notice that for $k=1$ the condition $\log |A| \ll  \dim (A)$ is obviously satisfies).
    Now let $\Lambda \subseteq A$ be a $k$--dissociated set such that $d:= \dim_k (A) = |\Lambda|$. 
    We split $\Lambda$ onto $m$ parts $\Lambda_j$ such that $|\Lambda_j|\ge d/(2m)$. 
    Put 
    $S = [k] \cdot \Lambda_1 + \dots + [k] \cdot \Lambda_m$.
    %$S= m\Lambda$.
    We have $S\subseteq m([k] \cdot A)$ and for an arbitrary  positive integer $n\le d/4m$, considering for any $j$ the sums $k_1 \la_1 +\dots +k_n \la_n$, where $\la_1,\dots, \la_n \in \Lambda_j$ are distinct 
    %numbers, 
    elements and $k_i$, $i\in [n]$ run over $[k]$, 
    we obtain  for an arbitrary positive integer $n$ 
    %one has 
\begin{equation}\label{f:nS_below} 
    |nS|\ge \prod_{j=1}^m \frac{k^n |\Lambda_j|^n}{2^n n!}
    \ge n^{-nm} (4m)^{-nm} (kd)^{nm} 
    %\,.
\end{equation}
    thanks to $k$--dissociativity of $\Lambda$. 
    %and, clearly, 
    %$|S| = \prod_{j=1}^m |\Lambda_j| \ge d^m (2m)^{-m}$.  
    %$|S| \ge d^m (2^m m!)^{-1}$. 
    Using the Pl\"unnecke inequality \eqref{f:Pl-R} 
    %and Theorem \ref{t:Rudin}, we obtain for any positive integer $n$
    and our bound \eqref{f:nS_below}, we obtain 
\[
%    \frac{d^{2nm}}{(4Cnm^3)^{nm} d^{nm}} \le \frac{|S|^{2n}}{\T_{nm} (\Lambda)} 
    n^{-nm} (4m)^{-nm} (kd)^{nm} \le |nS| \le |nm([k] \cdot A)|
    %\le |nmkA| 
    \le K^{n} |Z|
\]
    and hence choosing $n$ optimally, that is, $n = c\log |Z|/m$ where $c \in (0,1]$ is a small constant (the condition $m\ll \log |Z|$ guaranties that $n$ can be chosen as an integer), we get 
\[
    kd \le 4 nm K^{1/m} |Z|^{1/nm} \ll K^{1/m} \log |Z| \,.
\]
    It remains to check the condition $4nm = 4c \log |Z| \le d$ but our assumption $d= \dim_k (A) \ge \log |Z|$ and sufficiently small $c$  guarantee  this. 
    Finally, to get \eqref{f:dL_nA+} just choose in \eqref{f:nS_below}  the parameters $n=1$, $Z=A$, $m\to m-1$ and hence  
    $(kd)^{m-1} (4m)^{-(m-1)} \le |S| \le %K|A| = |mA|$ 
    |m([k] \cdot A)|$ 
    as required. 
%
%
\begin{comment} 
    Now to see how the quantity $\dim(A)$ can be replaced to $k\cdot \dim_k (A)$ just take  $\Lambda^{(k)}\subseteq A$, $|\Lambda^{(k)}|=\dim_k (A) :=d_k$, $\Lambda^{(k)}$ is a $k$--dissociated set and consider $S_k = \Lambda_1+\dots+\Lambda_m$, where $\Lambda_j \subseteq \Lambda^{(k)}$ and any element of $\Lambda^{(k)}$ belongs to at most $k$ sets $\Lambda_j$, $j\in [m]$. Of course such sets are exist and moreover, one can choose $|\Lambda_j| \gg d_k k/m$. 
    As above we take $n^* \Lambda_j$ and by construction any element of  $s\in nS$ can be written as $s=\sum_{\lambda \in \Lambda^{(k)}} \lambda n_\lambda$, where $0\le n_\lambda \le k$ for all $\lambda \in \Lambda^{(k)}$.
    After that repeat all calculations above. 
    This completes the proof. 
\end{comment} 
$\hfill\Box$    
\end{proof}
%\end{comment} 

%\bp 

\begin{remark}
    Of course we need condition 
    %$m\ll \log |A|$ 
    $m\ll \dim (A)$ 
    in \eqref{f:dL_nA}  (for simplicity let $k=1$) 
    because otherwise 
    the trivial upper bound $|mA|
    \le \binom{|A|+m-1}{|A|-1}$ and any set $A$ with $\dim (A) = |A|$ gives a contradiction. 
    Also, the same example shows that inequality \eqref{f:dL_nA+}, as well as the dependence on $\dim_k (A)$ in the right--hand side  of estimate \eqref{f:LY_2_particular} below 
    are tight up to some constants.  
\end{remark}

%\bp 

Inequality \eqref{f:dL_nA+} works for $m \ll \dim (A)$ (let $k=1$ for simplicity). 
We need a result on higher sumsets of $mA$ in terms of some dimensions of $A$, which works for  $m \gg \dim (A)$ if for a certain (large) number  $k$ the quantity  $\dim_k (A) \log \dim_k (A)$ is much greater than $\dim (A)$ (compare with Lemma \ref{l:dim_compare+}).
We give a sketch of the proof of estimate \eqref{f:LY} below, details are contained in \cite[Theorem 1]{LY}.

\begin{lemma}
%    Let $l = \dim(A)/\log (\dim (A))$. 
    Let $A\subseteq \Gr$ be a set, $l$ be a positive integer parameter, and  $\Lambda_l \subseteq A$ be any $l$--dissociated set.
    Then 
\begin{equation}\label{f:LY}
    |\Lambda_l| \log |\Lambda_l| \gg 
    \dim (\Sigma_{|\Lambda_l|} (\Lambda_l) ) 
    \gg 
            \min\{ |\Lambda_l| \log |\Lambda_l|, l\} \,.
\end{equation}
    In particular, for any integer $l$ 
    %one has 
    the following holds 
\begin{equation}\label{f:LY_cor}
    \dim (\Sigma_{\dim_l (A)} (A)) \gg \min\{ \dim_l (A) \cdot \log \dim_l (A), l\} \,.
\end{equation} 
    Also, for $k=\dim (A) \log \dim (A)$ and $m\ll \dim_k (A) \cdot \log \dim_k (A)$ one has 
%\begin{equation}\label{f:LY_2}
%    |\dim^2 (A) \log \dim (A)\, A| \ge 
%    \exp (\Omega(\dim_k (A) \cdot \log \dim (A)) ) \,,
%\end{equation}
\begin{equation}\label{f:LY_2}
    |m\, \Sigma_{\dim_k (A)} (A)| \ge 
    \exp (\Omega(m \log (m^{-1} \dim_k (A) \cdot \log \dim_k (A))) ) \,.
\end{equation}
    In particular, 
\begin{equation}\label{f:LY_2_particular}
    |\dim_k (A) \log \dim_k (A)\, \Sigma_{\dim_k (A)} (A)| \ge 
    \exp (\Omega(\dim_k (A) \cdot \log \dim_k (A)) ) \,. 
\end{equation}
    Finally, there is an absolute constant $C>0$ such that for all $m\le C^{-1} \log |\Sigma_{\dim_k (A)} (A)|$, we have 
    %the following holds 
\[
    |m\, \Sigma_{\dim_k (A)} (A)|\ge |\Sigma_{\dim_k (A)} (A)| \cdot 
    \left( \frac{\log \dim_k (A)}{C \log (e|A|\dim^{-1}_k (A))} 
    \right)^{m-1} \,.
\] 
\label{l:LY}
\end{lemma}
\begin{proof} 
    Let us start with the second inequality in \eqref{f:LY} (the first  one is not so interesting and follows from inequality %below  
    \eqref{f:dim_sum2}, say). 
    Let $n= |\Lambda_l|$, $Q=\Sigma_{n} (\Lambda_l)$ and $m=\dim (Q)+1$. 
    Suppose that $n>(2\log 3 + o(1))m/\log m$ and $l\ge m$. 
    Consider $n\times m$ matrix $M$ with entries equal $0,1$. 
    In other words, we take a set $D\subseteq \{0,1\}^m$, $|D|=n$ and we construct  this set  choosing at random and independently of each other $n$ vectors from the set $\{0,1 \}^m$. 
    Using the union bound one can  show that for any $\eps \in \{0,\pm 1\}^m \setminus \{0\}$ there is $d\in D$ such that $d$ is not orthogonal to $\eps$ (see \cite[Theorem 1]{LY}). 
    In other words, using the fact that $\Lambda_l$ is a $l$--dissociated set, $l\ge m$, we will see that the set of $m$ columns  of $M$ corresponds to a {\it dissociated} set $S\subseteq Q$, $|S|=m>\dim (Q)$ by the rule: $c = (c_1,\dots,c_n)\in M \to c_1+\dots+c_n := s_c \in S$. It will give  a contradiction and hence either 
    $m\gg n \log n$ or $m<l$ 
    as required.
    Indeed, put $\Lambda_l = \{\la_1,\dots, \la_n\}$ and consider an additive $n$--tuple with elements $s_c$.
    In other words, we take $\eps \in \{0, \pm 1\}^m \setminus \{0\}$ such that 
    \[
        0 = \sum_{j=1}^m \eps_j s_{c^{(j)}}
        = \sum_{i=1}^n \la_i \left( \sum_{j=1}^m \eps_j c^{(j)}_i \right) 
    \]
    and thus by $l$--dissociativity of $\Lambda_l$, we have 
    $\sum_{j=1}^m \eps_j c^{(j)}_i = 0$ for any $i\in [n]$ and this contradicts to our construction of the set $D$.

    Now we are ready to get \eqref{f:LY_2_particular} (the proof of \eqref{f:LY_2} is similar). 
    We know that  
    $k=\dim (A) \log \dim (A)\ge \dim_k (A) \log \dim_k (A)$ and using \eqref{f:LY_cor} with $l=k$, we obtain  
    $\dim (\Sigma_{\dim_k (A)} (A)) \gg \dim_k (A) \cdot \log \dim_k (A)$. 
    Putting $Q = \Sigma_{\dim_k (A)} (A)$ and applying inequality \eqref{f:dL_nA+} with $A=Q$ and $m\gg \dim (Q)$, we derive
\[
 %   |\dim_k (A) \log \dim (A)\, Q| \ge 
   |\dim_k (A) \log \dim_k (A)\, Q| \ge |mQ| \gg 
%\]
%\[
%   \gg 
   \exp (\Omega(\dim_k (A) \cdot \log \dim_k (A)) ) \,.
\]
%\[
%    \gg \exp (\Omega(\dim_k (A) \cdot \log \dim_k (A)) ) \,.
%\]
%    The last estimate follows from \eqref{f:dim_compare+} and our choice of the parameter $l$. 
    Finally, if we apply inequality \eqref{f:dL_nA} instead of \eqref{f:dL_nA+}, we get for $m\le C^{-1} \log |Q|$
\[
    |mQ| \ge |Q| \cdot \left( \frac{\dim (Q)}{C \log |Q|} \right)^{m-1}
    \ge 
    |Q| \cdot 
    \left( \frac{\dim_k (A) \cdot \log \dim_k (A)}{C \log |Q|} \right)^{m-1}
    \ge 
\]
\[ 
    \ge 
   |Q| \cdot 
    \left( \frac{\log \dim_k (A)}{C \log (e|A|\dim^{-1}_k (A))} 
    \right)^{m-1} \,.
\]
    Here we have used  a trivial upper bound for size of $Q$, namely, 
\[
    \log |Q| \le \log \binom{|A|}{\dim_k (A)} \le \dim_k (A) \log (e|A|\dim^{-1}_k (A)) \,.
\]    
This completes the proof. 
$\hfill\Box$    
\end{proof}

\bp 

To obtain Theorem \ref{t:growth_nA_intr} from the introduction just combine formulae \eqref{f:dL_nA}, \eqref{f:dL_nA+} of Lemma \ref{l:dL} (with $k=1$) and inequality \eqref{f:LY_2_particular} of Lemma \ref{l:LY}. In the later case it is sufficient to use the trivial bound  
$$|\dim_k (A) \log \dim_k (A) \Sigma_{\dim_k (A)} (A)|  \le \dim_k^2 (A) \log \dim_k (A) |\dim_k^2 (A) \log \dim_k (A)\, A| \,.$$ 
Another way to show the same (for small $k$) is the following. 
We know that if $0\in A$, then  $\Sigma_{\dim_k (A)} (A) = \dim_k (A)\, A$ but by Theorem \ref{t:dim(A+x)} below any shift of $A$ does not change the dimension too much.  

\bp 

%Now we 
The result above allows us to  
obtain a more delicate connection between $d^*(A)$ and $\dim_k (A)$ for large $k$. 
It gives in particular, that 
\begin{equation}\label{f:equivalence_d_dim}
    \dim_k (A) \ll d^* (A) \le d(A) \le \dim (A) 
\end{equation}
for  $k= \dim (A) \log \dim (A)$.

\begin{corollary}
    Let $A\subseteq \Gr$ be a set, $k= \dim (A) \log \dim (A)$ and $l$ be a positive integer. 
    Then 
\begin{equation}\label{f:d(A)_d_k}
    \dim_k (A) \ll d^*_l (A) \left(1+\frac{\log l}{\log \dim_k (A)}\right) \,.
\end{equation}
\label{l:d(A)_d_k}
\end{corollary}
\begin{proof} 
    Let 
    %$S\subseteq A$ 
    $S\subseteq \Gr$ 
    be a set such that $A\subseteq \Span_l (S)$ and $|S| = d^*_l (A)$. 
    Thus for any $n$ one has $|nA| \le \Span_{nl} (S) \le (2nl+1)^{|S|}$. 
    We choose $n= \dim^2_k (A) \log \dim_k (A)$. 
 Combining this bound with \eqref{f:LY_2_particular}, we get
\[
(2nl+1)^{|S|}
    \ge 
    |\dim_k (A) \log \dim_k (A)\, \Sigma_{\dim_k (A)} (A)| \ge 
    \exp (\Omega(\dim_k (A) \cdot \log \dim_k (A)) ) \,, 
\]
    and hence
\[
    |S| \log (l\dim_k (A)) \gg \dim_k (A) \cdot \log \dim_k (A) \,.
\]
This completes the proof. 
$\hfill\Box$    
\end{proof}

\bp 

Lemma \ref{l:LY} 
%above 
allows us to characterize all combinatorial cubes having the property that $\dim (Q(\Lambda)) \ll \dim(A)$ for a sufficiently dissociated set $\Lambda$.
This characterisation is possible in terms of the  dimension  $\dim_k (A)$ for a certain large $k$. 
The answer is that 
%\dim_k (A)$ must be 
$\dim (Q(\Lambda)) \sim \dim (A)$ iff 
$\dim_k (A)\sim \dim (A)/ \log \dim (A)$.
%with $k=\dim (A) \log \dim (A)$.

\begin{corollary}
    %Let $\mathcal{R}$ be a ring, and $A\subseteq \mathcal{R}$ be a set. Also, let  $k=\dim (A) \log \dim (A)$,  $\Lambda_k$ be a maximal $k$--dissociated subset of $A$, and suppose that all numbers $j\in [k]$ are invertible in $\mathcal{R}$.\\ 
    Let $\Gr$ be a group, and $A\subseteq \Gr$ be a set.
    Also, let  $k=\dim (A) \log \dim (A)$, and  $\Lambda_k$ be a maximal $k$--dissociated subset of $A$.\\ 
    If $\dim (Q(\Lambda_k)) \le K \dim (A)$, then $\dim_{k} (A) \ll K \dim (A)/ \log \dim (A)$.\\ 
    On the other hand, if $\dim_k (A) \le K \dim (A)/ \log \dim (A)$, then $\dim (Q(\Lambda_k)) \ll K \dim (A)$. 
\label{c:Q(A)_A}
\end{corollary}
\begin{proof} 
%    Let  $\Lambda_k$ be a maximal $k$--dissociated subset of $A$. 
    From formula \eqref{f:LY} of Lemma \ref{l:LY} with $l=k$,  we see that 
    \begin{equation}\label{tmp:09.05_2}
       K \dim (A) 
       %\ge \dim (Q(\Lambda)) 
       \ge \dim (Q (\Lambda_k)) \gg \dim_k (A) \log \dim_k (A) 
       \gg 
       \dim (A) 
    \end{equation}
    as required. 
    We have applied Lemma \ref{l:dim_compare+} to derive the last inequality in \eqref{tmp:09.05_2} and thus we need to assume that $\Gr$ is a ring such that all numbers $j\in [k]$ are invertible in $\Gr$.
    Rigorously speaking,  we do not need in this implication to obtain Corollary \ref{c:Q(A)_A}.

    We now take a dissociated set $\Lambda_* \subseteq Q(\Lambda_k)$, $|\Lambda_*| = \dim (Q(\Lambda))$.
    %and let $\Lambda_k$ is defined as above. 
    %As in 
    Similarly to 
    formula \eqref{tmp:dim_compare+} of Lemma \ref{l:dim_compare+}, we have
\[
    2^{|\Lambda_*|}
    %\le |\Span_{|\Lambda_*|} (\Lambda_k)| 
%    \le \sum_{j=1}^k |\Span_{|\Lambda||\Lambda_*| k} (\Lambda_k)|
%    \le 
%    k (2k|\Lambda||\Lambda_*| +1)^{\dim_k (A)}
    \le |\Span_{|\Lambda_*|} (\Lambda_k)| 
    \le (2|\Lambda_*| +1)^{\dim_k (A)}
\]
    or, in other words, 
\[
%    |\Lambda_*| \ll \dim_k (A) \cdot \log (k|\Lambda||\Lambda_*|) \ll \frac{K \dim (A)}{\log \dim (A)} \cdot \log (|\Lambda||\Lambda_*|) \,.
    |\Lambda_*| \ll \dim_k (A) \cdot \log |\Lambda_*| \le \frac{K \dim (A)}{\log \dim (A)} \cdot \log |\Lambda_*| \,.
\]
    It gives us $|\Lambda_*| \ll K\dim (A)$
    %$(1+\log_{\dim(A)} (K))$
    because, trivially,  we can assume that $K\ll \log \dim (A)$. 
This completes the proof. 
$\hfill\Box$    
\end{proof}

%\bigskip 

\bp

We now study the question how to calculate the dimensions of sumsets. 
The estimate 
%Let us remark that (see \cite[Section 8]{ss_dim})
%\begin{equation}\label{f:dim_sum_1}
%    \dim(A+S) \lesssim \dim (A) + |S| \,,
%\end{equation}
%and 
\begin{equation}\label{f:dim_sum*}
    \dim(A+B) \lesssim \dim (A) + \dim (B) 
\end{equation}
for any 
%disjoint 
sets $A,B \subseteq \Gr$ is, basically, contained in \cite[Section 8]{ss_dim}. 
Here the sign $\lesssim$ shows that the factor $\log \dim (A+B)$ is allowable. 
For the sake of completeness we give the proof of inequality \eqref{f:dim_sum*} in Lemma \ref{l:dim_sum} below see, e.g., estimate \eqref{f:dim_sum1}.
%
%
%Now we obtain, in particular, estimate \eqref{f:dim_sum_2}. 
%For small $k$ 
%Generally speaking, 
Let us remark that, in general, 
formula \eqref{f:dim_sum1} below is tight.
Indeed,  take the  cube $Q = \{0,1\}\cdot a_1 + \dots + \{0,1\}\cdot a_d$ from Example \ref{exm:cube_with_large_dim}. 
Then $\dim (Q) \sim d \log d$ and it is  larger by a logarithm, then $\sum_{j=1}^d \dim (\{0,1\}\cdot a_j) = d$.  Finally, taking a sufficiently sparse set  $A$ with $|kA| = \binom{|A|+k-1}{|A|-1}$, we see that inequality \eqref{f:dim_sum2} is tight as well (take $l=1$ and $k = |A|^2$, say, and use that $\dim (kA) \gg \log |kA|$).

\begin{lemma}
    Let 
    $\mathcal{R}$ be a ring, 
    %$\Gr$ be an abelian group, 
    $k$, $l$ be positive integers, and $C_1,\dots,C_k \subseteq \mathcal{R}$ be any sets.
    Suppose that all numbers $j\in [l]$ are invertible in $\mathcal{R}$.
    Then
\begin{equation}\label{f:dim_sum1}
    \dim_l (C_1 + \dots + C_k) 
    \ll
    l \sum_{j=1}^k \dim_l (C_j) \cdot \log_{l+1} (k \sum_{j=1}^k \dim_l (C_j)) \,,
\end{equation}
    and for any $A\subseteq \mathcal{R}$ one has 
\begin{equation}\label{f:dim_sum2}
    \dim_l (kA) \le \dim_l (\Sigma_{k} (A) ) \ll d^*_l (A) \log_{l+1} (k d^*_l (A)) \ll l \dim_l (A) \log_{l+1} (k \dim_l  (A))
    \,.
\end{equation}
\label{l:dim_sum}
\end{lemma}
\begin{proof} 
    We begin with \eqref{f:dim_sum1}. 
    Let $d^*_l (C_j) = |S_j|$ such that $C_j \subseteq \Span_l (S_j)$, $j\in [k]$. 
    Put $S = \bigcup_{j=1}^k S_j$ and let $C= C_1 + \dots + C_k$. 
    Then it is easy to see that $C \subseteq \Span_{kl} (S)$. 
    In other words, thanks to \eqref{f:d^*_k_addition} and \eqref{f:d_dim}, we have 
    \[
%        \frac{\dim_{kl} (C)}{kl\log_{kl} (\dim_{kl} (C))} 
    \frac{\dim_{l} (C)}{l\log_{l+1} (kl \dim_{l} (C))}
        \ll
        \frac{d^*_{kl} (C)}{l} \le \frac{1}{l} \sum_{j=1}^k d^*_{l} (C_j) 
        %\le \sum_{j=1}^k \dim_{kl} (C_j) 
        \le
        \sum_{j=1}^k \dim_{l} (C_j)
        %\,.
    \]
    and hence 
\[
    \dim_{l} (C) 
    \ll
    l \sum_{j=1}^k \dim_l (C_j) \cdot \log_{l+1} (k\sum_{j=1}^k \dim_l  (C_j)) 
    %\,.
\] 
    as required. 
%    It remains to use bound \eqref{f:dim_compare+} of Lemma \ref{l:dim_compare+} to estimate  $\dim_l(C)$ via $\dim_{kl} (C)$. 
%    as required. 

    To obtain \eqref{f:dim_sum2} just take a minimal set 
    %$\Lambda \subseteq A$ 
    $\Lambda \subseteq \Gr$ 
    such that $A\subseteq \Span_l (\Lambda)$ and notice that $mA \subseteq \Span_{kl} (\Lambda)$ for all $m\le k$. It implies $d^*_{kl} (\Sigma_{k} (A)) \le |\Lambda| = d^*_l (A)$.
    Applying the second inequality in \eqref{f:d_dim} with $l=l$ and $k=kl$, we get 
    $$
        \dim_{l} (\Sigma_{k} (A)) \ll d^*_{kl} (\Sigma_{k} (A) ) \log_{l+1} (kl d^*_{kl} (\Sigma_{k} (A) )) 
        \le d^*_l (A) \log_{l+1} (kl d^*_l (A)) 
    $$
    $$
        \le l\dim_l (A) \log_{l+1} (kl^2 \dim_l  (A))
        \ll 
        l\dim_l (A) \log_{l+1} (k \dim_l  (A))\,.
    $$
    This completes the proof. 
$\hfill\Box$    
\end{proof}

\bp

%\begin{remark}
%    Suppose that the number of sets $C_j$ in Lemma \ref{l:dim_sum} is large, namely, $k\gg \dim (C_1+\dots+C_k) \log \dim (C_1+\dots+C_k)$. Then in view of equivalence \eqref{f:equivalence_d_dim} \eqref{f:dim_sum1} 
%\end{remark}

%\bp

%\bp 

Upper bounds for dimensions of sets with really small doubling can be found in \cite[Theorem 4.2]{ss_dim}.  
%We show that 
Here we give a similar (and slightly sharper) result for $\dim_k (A)$. 
%%Notice that, say, the second part of Example \ref{exm:1_2} shows that the result below does not hold for large $K$. 

\begin{theorem}
    Let $A\subseteq \Gr$ be a set and $k$ be a positive integer.
    Suppose that $|A+A|\le K|A|$ and $K\le \dim_k (A)$. 
    Then 
\begin{equation}\label{f:d_k_small_doubling}
    \dim_k (A) -\log_{k+1} |A| 
        \ll 
            \frac{K \log^6 (2K) \log \log (4K)}{\log k} \log \left( \frac{K}{\log k} \right) \,.
\end{equation}
    If $\mathcal{R}$ is a ring such that all elements of $[k]$ are ivertible, $k \le |A|$ and 
    %and for any 
    $K$ is an arbitrary number, then  
    %one has 
\begin{equation}\label{f:d_k_small_doubling1}
    \dim_k (A) \ll (\log_k |A| \log (k\log_k |A|) + K \log^6 (2K) \log \log (4K)) \cdot \log_k (K\log_k |A|) \,.
\end{equation}
    Now let $k_* = \dim (A) \log \dim (A)$ and $K\ge 1$ be an arbitrary number. 
    Then 
\begin{equation}\label{f:d_k_small_doubling2}
    \dim_{k_*} (A) \ll 
        \frac{\log |A|}{\log \log |A|} + K \log^6 (2K) \log \log (4K) \,.
\end{equation}    
\label{t:d_k_small_doubling}
\end{theorem}
\begin{proof} 
    Let $\Lambda = \{ \la_1, \dots, \la_d\} \subseteq A$ be a $k$--dissociated set, $d=|\Lambda|=\dim_k (A)$. 
    Consider $(k+1)^d$ distinct sums of the form $k_1 \la_1 +\dots +k_d \la_d$, where $k_j\in \{0,1,\dots, k\}$. 
    Using \cite[Lemma 4.1]{ss_dim} (which is a consequence of Sander's result from \cite{Sanders_Bog}), we get for $K\le d$ 
\begin{equation}\label{tmp:21.04_1}
    (k+1)^d \le |dA| \le \left( \frac{3ed}{K} \right)^{O(K \log^6 (2K) \log \log (4K))} |A| 
\end{equation}
    and hence
\[
    d\log (k+1) - \log |A| \ll  K \log^6 (2K) \log \log (4K) \cdot \log (d/K) 
    \le 
     K \log^6 (2K) \log \log (4K) \cdot \log d 
\]
as required.

Now let us prove \eqref{f:d_k_small_doubling1}. 
By Chang's lemma \cite{c1} and  \cite[Lemma 4.1]{ss_dim}, we have 
\begin{equation}\label{tmp:23.04_1}
    A \subseteq P-P + (S_1-S_1) + \dots + (S_l-S_l) \,,
\end{equation}
    where $S_j$, $P$ are disjoint sets, $|S_j| \le 2K$, $j\in [l]$, $l\ll \log^6 (2K) \log \log (4K)$,
    the sum $S_1+\dots+S_l+P$ is direct and $|P| \gg |A| \cdot \exp(- \log^6 (2K) \log \log (4K))$ is a proper generalized arithmetic progression of dimension $d=O(\log^6 K)$.  
    
    Clearly, for any $j\in [l]$ one has $\dim_k (S_j) \le 2K$
    and thus by formula \eqref{f:d_k_addition} the following holds  
\begin{equation}\label{tmp:25.04_1}
    d^*_k (P+S_1+\dots+S_l) \le d_k (P) + \sum_{j=1}^l d_k (S_j) \le d_k (P) + 2Kl \,.
\end{equation}
    Also, writing $P=P_1+\dots+P_d$, we have $d^*_k (P) \ll  \sum_{j=1}^d \log_{k+1} |P_j| = \log_{k+1} |P|$ and one can see that 
    $d_k (P) \ll d_k^* (A) \cdot \log (kd^*_k (A)) \ll \log_k |P| \log (k\log_k |P|)$ by 
    the second inequality from 
    %\eqref{f:d_dim}.
    \eqref{f:d_dim+}.
    Using Lemma \ref{l:dim_sum}, the second inequality from \eqref{f:d_dim} and the obtained  bound \eqref{tmp:25.04_1}, we get 
    %the previous calculation, we get 
%    \[
%        \dim_k ((S_1-S_1) + \dots + (S_l-S_l)) \ll \log K 
%        \dim_k (S_1 + \dots + S_l) \ll Kl \log^2 K
%    \] 
\[
    \dim_k (A) \ll (\log_k |P| \log (k\log_k |P|) +  Kl) (\log_k \log_k |P| + \log_k K)
    \le 
\]
\[
    \le 
    (\log_k |A| \log (k\log_k |A|) + K \log^6 (2K) \log \log (4K)) \cdot \log_k (K\log_k |A|)
\]
    as required.

    It remains to obtain \eqref{f:d_k_small_doubling2}. 
    Put $d_* = \dim_{k_*} (A)$ and notice that if $K \ge d^2_* \log d_*$, then the result is trivial. 
    Using formula \eqref{f:LY_2_particular} of Lemma \ref{l:LY} as well \cite[Lemma 4.1]{ss_dim} as in \eqref{tmp:21.04_1}, we get
\[
    \exp (\Omega(\dim_{k_*} (A) \cdot \log \dim_{k_*} (A)) ) 
    \le 
    |d_* \log d_*\, \Sigma_{d_*} (A)| 
%    \le 
\]
\[
    \le 
    d^2_* \log d_* \cdot 
    \left( \frac{3e d^2_* \log d_*}{K} \right)^{O(K \log^6 (2K) \log \log (4K))} |A| 
\] 
    or, in other words, 
\[
    d_* \cdot \log d_* 
    \ll 
    K \log^6 (2K) \log \log (4K) \cdot \log d_* + \log |A| \,.
\]
    It gives us 
\[
    \dim_{k_*} (A) \ll \frac{\log |A|}{\log \log |A|} + 
    K \log^6 (2K) \log \log (4K) \,.
\] 
    This completes the proof. 
$\hfill\Box$    
\end{proof}

\begin{remark}
    As Example \ref{exm:zoo} (part two) shows one can obtain an analogue of  Theorem \ref{t:d_k_small_doubling} for sums of sets with small additive doubling. 
\end{remark}

%\bp 

In the next result we show that polynomial growth and the additive dimension are closely connected to each other (up to some logarithms).

\begin{proposition}
    Let $k$ be a positive integer, and $\Gr$ be an abelian group.\\  
    If 
    %$A\subseteq \mathcal{R}$ 
    $A\subseteq \Gr$ 
    is a set of polynomial growth $|nA|\le n^d |A|$, then 
    $\dim_k (A) \ll d \log_{k+1} d + \log_{k+1} |A|$.\\ 
    %\cdot |A|^{O(1/d \log d)}$.  
    %(in particular, if $d\gg \log |A|$, then $\dim(A) \ll d$).
    Conversely, let $\mathcal{R}$ be a ring such that all numbers $j\in [k]$ are invertible. 
    Then any set $A\subseteq \mathcal{R}$ has  polynomial growth $|nA|\le n^d |A|$ with $d\ll \dim_k (A) \log (k+1)$.\\
%
%    In particular, if $d \gg \log |A|$, then  the quantities $d$ and $\dim (A)$ are equivalent up to some multiplicative constants. 
    %$d\ll \dim (A) \ll d$. 
\label{p:pol_growth}
\end{proposition}
\begin{proof}
    Put $L=\log |A|$.
    Let $\Lambda \subseteq A$ be a maximal $k$--dissociated subset of $A$, $|\Lambda| = \dim_k (A)$.  
    We obtain even two bounds for  $\dim_k (A)$. 
    By Theorem \ref{t:Rudin} (or simple counting argument), we have for a certain  absolute constant $C>0$
\begin{equation}\label{fi:pol_growth_1}
    \frac{|\Lambda|^n} {(Cn)^n} \le |n \Lambda| \le |nA| \le n^d |A| \,.
\end{equation}
    Taking $n\sim d \log d$, we obtain $|\Lambda| \ll d \log d \cdot  |A|^{O(1/d \log d)}$. 
    This 
    %argument 
    calculation 
    does not use the fact that $k\ge 1$ and to do this we apply the argument from the proof of Lemma \ref{l:dim_compare+}. 
    By $k$--dissociativity of $\Lambda$ all sums $\sum_{\la \in \Lambda} n_\la \la$, where $n_\la \in [0,1,\dots,k]$ are distinct and hence 
\begin{equation}\label{fi:pol_growth_1.5}
    (k+1)^{\dim_k (A)} \le |k\dim_k (A) A| \le (k\dim_k (A))^{d} |A|
\end{equation}
    and the result follows.

    Conversely, as in the proof of  Lemma \ref{l:dim_compare+}, we have $A \subseteq \bigcup_{j\in [k]} j^{-1} \Span_k (\Lambda) := Q$, $|\Lambda| = \dim_k (A)$ and hence (we can assume that $n\ge 2$) 
\begin{equation}\label{fi:pol_growth_2}
    |nA| \le |nQ| \le k (2nk+1)^{|\Lambda|} \le n^d \le n^d |A| \,, 
\end{equation}
    where $d = O( \log (k+1) \cdot \dim_k (A))$, say. 
%
%    It remains to obtain \eqref{f:pol_growth}. If $|nA| \ll n^{O_K (L/\log L)}$, then, clearly, $\dim (A) \ll_K L$ by the same argument as in \eqref{fi:pol_growth_1}. Now suppose that $\dim (A) = O_K(\log |A|)$. 
%
    This completes the proof. 
$\hfill\Box$    
\end{proof}

\bp

Also, we show that any set with small dimension has a rather dense Freiman model (see \cite[Section 5.3]{TV}). 
For example, if $\dim_k (A) \ll \log_k |A|$, then Proposition \ref{p:Freiman_dim} below gives us a set  $B \subseteq \Z/m\Z$ with $m = |A|^{O(\log_k (kl))}$ such that $B$ is an isomorphic image of a large part of $A$.

\begin{proposition}
    Let $\mathcal{R}$ be a ring, $A\subseteq \mathcal{R}$ be a set, $k,l$ be positive integers, 
    and $m \ge k(4kl +1)^{\dim_k (A)}$. 
    Suppose that all numbers $j\in [k]$ are invertible. 
    Then there is $A_* \subseteq A$ and $B \subseteq \Z/m\Z$ such that $|A_*|\ge |A|/l$ and $A_*$ is $l$--isomorphic to $B$.
\label{p:Freiman_dim}
\end{proposition}
\begin{proof}
    We follow the standard argument of Ruzsa  see, e.g., \cite[Lemma 5.26]{TV}. In other words, we need to estimate the size of $|lA-lA|:=m$. 
    In terms of $\dim_k (A)$ it gives us 
    %we obtain as in the proof of Lemma \ref{l:dim_compare+} that
    (consult the proof of Lemma \ref{l:dim_compare+}) 
\[
    |lA - lA| \le k (4kl +1)^{\dim_k (A)} \,.
\]
    This completes the proof. 
$\hfill\Box$ 
\end{proof}

%\bp 

\bp

Finally,  we consider a rather important case when a set $A$ stops growing under addition and  we give a criterion of this absent of the growth in terms of $\dim_k (A)$ for a certain $k$ or, equivalently, in terms of the set $Q(\Lambda_k)$, thanks to Corollary \ref{c:Q(A)_A}.

\begin{theorem}
    Let $A\subseteq \Gr$ be a set, $k=\dim (A) \log \dim (A)$ and $K\ge 1$ be a real number.\\
    If $|nA| \le |A|^K$ for $n \ll K^2 \log^2 |A| \log (K\log |A|)$, 
    then $\dim_k (A) \ll K \log |A| / \log (K\log |A|)$.\\
    Now suppose that $\mathcal{R}$ is  a ring such that all numbers $j\in [k]$ are invertible and  
    $\dim_k (A) \le K \log |A| / \log (K\log |A|)$.
    Then $|nA| = |A|^{O(K)}$ for all $n = \dim^{O(1)} (A)$. 
\label{t:A^C_growth}
\end{theorem}
\begin{proof} 
    Define $N = O(K^2 \log^2 |A| \log (K\log |A|))$ such that by our  assumptions for all $n\le N$ 
    %one has 
    the following holds 
    $|nA| \le |A|^K$. 
    From $|nA| \le |A|^K$, it follows that $\dim (A) \ll K \log |A|$ (see, e.g.,  calculations in \eqref{fi:pol_growth_1}). 
    We use bound  \eqref{f:LY_2_particular} of Lemma \ref{l:LY} 
    (also, it is possible to apply inequality \eqref{f:dL_nA+} of Lemma \ref{l:dL}) to derive 
\[    
    \dim^2 (A) \log \dim (A) |A|^K \ge \dim^2 (A) \log \dim (A) |NA| 
%    \ge 
\]
\[
    \ge
    \dim^2 (A) \log \dim (A)
    |\dim^2 (A) \log \dim (A)\, A| 
    \ge |\dim_k (A) \log \dim_k (A)\, \Sigma_{\dim_k (A)} (A)|
\]
\[
    \ge 
    \exp (\Omega(\dim_k (A) \cdot \log \dim_k (A)) )  
\] 
    and hence $\dim_k (A) \ll K \log |A| / \log (K\log |A|)$.

    Now suppose that  $\dim_k (A) \le K \log |A| / \log (K\log |A|)$.
    Using Lemma \ref{l:dim_compare+}, we obtain $\dim (A) \ll K \log |A|$.
    %but, unfortunately, this requires  that $\Gr$ is  a ring such that all numbers $j\in [k]$ are invertible in $\Gr$.
    %To obtain the same 
    Now take a $k$--dissociated set $\Lambda_k$ such that $|\Lambda_k| = \dim_k (A)$. 
    Using the arguments as in \eqref{tmp:dim_compare+} of Lemma \ref{l:dim_compare+}, we get 
\[
    |nA| \le k |\Span_{nk} (\Lambda_k)| \le k (2nk+1)^{\dim_k (A)}
    =
    \exp (\dim_k (A) \log (2nk+1) + \log k)
    =
    |A|^{O(K)} 
\]
    for all $n = \dim^{O(1)} (A)$.
    This completes the proof. 
$\hfill\Box$    
\end{proof}

\bp 

In the proof of Theorem \ref{t:A^C_growth} we have considered sets $A$ with $\dim (A) \ll K\log |A|$. Clearly, it is a much larger family of sets than just 
%the family of sets with 
having the property 
$|nA| = |A|^{O(K)}$ for all $n = \dim^{O(1)} (A)$. It is interesting to describe this family and 
%here 
below 
we give some examples. 

\bp 

{\bf Problem.} Let $\Gr$ be an abelian group. Characterise all sets $A \subseteq \Gr$ with $\dim (A) \ll \log |A|$.

\begin{exm} Let $\Gr$ be an abelian group and $A\subseteq \Gr$ be set. We write $L$ for $\log |A|$.\\ 
$1)~$ Let $H_j$, $j\in [K]$ be some  disjoint 
%sets of equal sizes with $|H_j+H_j| \le C |H_j|$ for a certain absolute constant $C\ge 1$ 
arithmetic progressions 
and  put $A (K) =\bigsqcup_{j} H_j$. Then $\dim (H_j) \ll \log |H_j| < L$ and $\dim (A (K)) \ll K L$.
Let us assume, in addition,  that the sum $H_1+\dots+H_K$ is direct. Then for all positive integers $n$ one has $|nA (K)| \le n^K |K A(K)| \ll  n^K (|A|/K)^K$. 
\\
$2)~$ Now let $\mathcal{A} = \sum_{i=1}^{t} A_i (K_*)$, where 
%the parameters 
$K_*, t=O(K^{1/2})$, $K = O(|A|)$, 
%will be chosen later 
and the sets $A_i (K_*)$ are constructed as in the previous example. 
%Put $Q=Q ( \bigcup_{i=1}^t A_i (\sqrt{K}) )$
\begin{comment}
Then in view of Corollary \ref{c:Q(A)_A}, we have 
\[
    \dim (\mathcal{A}) \le \dim (Q ( \bigcup_{i=1}^t A_i (K_*)) \ll K_* t \cdot  \dim (\bigcup_{i=1}^t A_i (K_*) )) 
    \ll 
        (K_* t)^2 L
        \ll K L \,, 
\]
provided $\dim_k (\mathcal{A}) \ll t K_* \dim (\mathcal{A})/\log \dim (\mathcal{A})$ with $k = \dim (\mathcal{A}) \log \dim (\mathcal{A})$. 
But using Theorem \ref{t:d_k_small_doubling} with sufficiently small doubling constant  $K$ (that is, the sumsets $H_j+H_j$ are really small), we obtain 
\[
    \dim_k (\mathcal{A}) \le \sum_{j=1}^t \dim_k (A_i (K_*)) \ll t K_* \log_k |A| \,.
\]
\end{comment} 
We need to estimate $\dim (\mathcal{A})$. 
One has 
$$
|n\mathcal{A}| \le \prod_{i=1}^t |n A_i(K_*)| \le n^{tK_*}\prod_{i=1}^t |K_* A_i(K_*)| 
\ll 
n^{tK_* } (|A|/K_*)^{tK_*} \,.
$$
Taking $n=\dim(\mathcal{A})$ and using the same argument as in the first part of the proof of Theorem \ref{t:d_k_small_doubling}, we obtain 
\[
    \dim(\mathcal{A}) \ll t K_* \log (\dim(\mathcal{A}) |A|) \ll K L \,,
\]
thanks to $K\ll |A|$. 
It gives us another example of a set with $\dim (A) \ll K L$.\\ 
$3)~$ In view of Proposition \ref{p:pol_growth} any set $A$ with polynomial growth $d=O(KL/\log L)$ has $\dim (A) \ll K L$ for small $K$. As in $2)$ we can take sums of such sets for small parameters  $t$ and $d$. 
\label{exm:zoo}
\end{exm}

\section{Additive dimensions and the quantity $\T_k$} 
\label{sec:T_k}

In this section we consider some 
%variations 
variants 
of the dimension $\dim (A)$, which are convenient for counting 
%the important quantity  
$\T_k (A)$. 
For simplicity we do not have to deal with the dimensions $\dim_l (A)$ or $d^*_l (A)$ for $l>1$ because they are connected with the other quantities, namely, with  $\T_k (r_{[l] \cdot A})$ see, e.g., formulae \eqref{f:dL_nA},  \eqref{f:dL_nA+} of Lemma \ref{l:dL}.

Let $\a \in (0,1]$ and $k\ge 2$ be an integer. Put 
\begin{equation}\label{def:dim1}
    \dim_{\a,k} (A) = \min \{ \dim (B) ~:~ B \subseteq A,\, \T_k (B) \ge \a \T_k (A) \} \,,
\end{equation} 
and 
\begin{equation}\label{def:dim2}
    \dim_\a (A) = \min_{k\ge 2} \dim_{\a,k} (A) 
\end{equation}
Clearly, $\dim_{\a,k} (A) \le \dim (A)$ for all $\a$ and $k$. 
Notice that $\dim(A)$ is subadditive and monotone  but  $\dim_{\a,k} (A)$, $\dim_\a (A)$ are not.

%\begin{corollary}
%\label{c:dim_A+x}
%\end{corollary}

\bp 

From the proof of Proposition \ref{p:pol_growth} and the H\"older inequality one has for any $k$ and $A\subseteq \Gr$ that 
\begin{equation}\label{f:}
    \dim (A) \ge \log_{2k+1} |kA| \ge \log_{2k+1} \left( \frac{|A|^{2k}}{\T_k (A)} \right) 
    \,.
\end{equation}
%Also, the same proof implies 
It implies in particular, 
%that there are lower bounds for $\T_k (B)$ for any set %$B\subseteq A$ in terms of $\dim (A)$, e.g.,
\begin{equation}\label{f:T_k(B)}
    \T_k (A) \ge \frac{|A|^{2k}}{(2k+1)^{\dim (A)}} 
    %\,.
\end{equation}
and hence there is a connection between $\dim (A)$ and $\T_k (A)$ (as well as with the size of the sumset $kA$). Below we obtain a stronger result.

%\bp 

%If $|kA|\ll |A|^C$ for all $k$ and a certain $C>1$, then, clearly, $\dim (A) \ll_C \log |A|$. Indeed, let $\Lambda$ be a maximal dissociated set such that $|\Lambda| = \dim (A)$. Then 
%\[
%    \frac{|\Lambda|^k}{k^k} \ll |k\Lambda| \le |kA| \ll |A|^C
%\]
%and taking $k\sim \log |A|$ we obtain  that $|\Lambda| \ll_C \log |A|$ as required. Below we generalize this simple observation and we show that our dimensions $\dim_{\a,k} (A)$ are closely connected with counting of the quantity $\T_k (A)$. 

%In view of Proposition \ref{p:pol_growth} we now that $\dim(A)$ and  polynomial growth are connected. 
%growth of sumsets of $A$

\begin{theorem}
    Let $\a\in (0,1]$ be a real number and 
    %$\dim_{\a,k} (A)$.
    %= K \log |A|$. 
    $k$ be a positive integer. 
    Then 
\begin{equation}\label{f:T_k_dim_1-}
    \T_k (A) \le \left( \frac{16 Ck }{\dim_{\a,k} (A) (1-\a^{1/2k})^2} \right)^{k} \cdot |A|^{2k}\,,
%    \quad \mbox{ and } \quad 
\end{equation} 
and 
\begin{equation}\label{f:T_k_dim_1}
    \T_k (A) \le \frac{C' |A|^2 \T_{k-1} (A)}{\dim_{\a,k} (A) (1-\a^{1/k})^2} 
    \cdot \log (|A|^{2k-2} \T^{-1}_{k-1} (A))
    \le 
    \frac{C' k |A|^2 \T_{k-1} (A) \log |A|}{\dim_{\a,k} (A) (1-\a^{1/k})^2}  
    \,,
\end{equation}
    where $C>0$ is an absolute constant as in Theorem \ref{t:Rudin}
    and $C'>0$ is another absolute constant.\\ 
    Conversely, 
    %for any $A\subseteq \Gr$ and an integer $k> \dim (A)$  one has 
    %$$
    %    \T_k (A) \ge |A|^{2k} \cdot (4k|A|)^{-\dim (A)} \,,
    %$$
%    and  
    writing $d=\dim (A) = M \log |A|$, we find 
    %for any $B\subseteq A$ 
    an integer $m$, $\log |A|/\log M \ll m \le \dim (A)/2$ such that 
    \begin{equation}\label{f:T_k_dim_2}
        \T_m (A) \ge \frac{|A|^{2m}}{2d 4^m e^d \binom{d+1}{2m}} \ge \frac{|A|^{2m}}{|A|^{C_* M}} \,, 
    \end{equation}
    where $C_*>0$ is an absolute constant. 
    In particular, 
    for any $B\subseteq A$ with $|B| = \d |A|$ and  $d:= \dim (B)$ one has 
    %%for a certain $\log |B|/\log M \ll m \le d$ that  
    \begin{equation}\label{f:T_k_dim_2'}
 %       d \gg \log (|B|^{2m} \T^{-1}_m (B))\,, 
 %   \quad \mbox{ and } \quad 
    d \log (d/\d) \gg \log (|A|^{2d+1} \T^{-1}_{d+1} (A))
    \,.
    \end{equation}
\label{t:T_k_dim}
\end{theorem}
\begin{proof}
    Let 
    $L=\log |A|$ and 
    $l$ be a parameter, which we will choose later. Then split $A$ as $A=(\bigsqcup_{j=1}^s A_j) \bigsqcup A_*$, where $A_j$ are dissociated, $|A_j| = l$ and $\dim (A_*)<l$. Clearly, $s\le |A|/l$. 
    By the norm property of $\T_k$ and Rudin's Theorem \ref{t:Rudin}, we have
\begin{equation}\label{tmp:02.05_1}
    \T^{1/2k}_k (A) \le \T^{1/2k}_k (A_*) + \sum_{j=1}^s \T^{1/2k}_k (A_j) \le \T^{1/2k}_k (A_*) + |A|/l \cdot (Ck)^{1/2} l^{1/2} \,.
\end{equation}
    Writing $\T_k (A) = \kappa^k |A|^{2k}$ and choosing $l=Ck \kappa^{-1} (1-\a^{1/2k})^{-2}$, we get  $\T_k (A_*) \ge \a \T_k (A)$.
    By the definition of the quantity $\dim_{\a,k} (A)$, we see that 
    $\dim_{\a,k} (A)\le Ck \kappa^{-1} (1-\a^{1/2k})^{-2}$ and hence we derive \eqref{f:T_k_dim_1-}.
%    the first estimate of \eqref{f:T_k_dim_1}. 

    The second bound \eqref{f:T_k_dim_1} can be obtained similarly to \cite[Proposition 5.3]{ss_dim}.
    Indeed,  put $\mathcal{E} = \bigsqcup_{j=1}^s A_j$ and 
    write $\T_k (A) = \o \T_{k-1} (A) |A|^2$, $\o \in (0,1]$.  
    Let $p= \log (|A|^{2k-2} \T^{-1}_{k-1} (A)) \le k \log |A|$, $p\ge \log |A|$ and  $l:= C' (1-\a^{1/k})^{-2} p \omega^{-1}$, where $C'>0$ is a sufficiently large  absolute constant. 
    Then 
\[
    \T_k (A) = \T_k (A,\dots, A, A_*, A_*) +  2\T_k (A,\dots, A, \mathcal{E}, A_*) +\T_k (A,\dots, A, \mathcal{E}, \mathcal{E})
    = \sigma_0 + \sigma_1 + \sigma_2 \,.
\]
    If $\sigma_0 \ge \a^{1/k} \T_k (A)$, then by the  H\"older inequality \eqref{f:T_k_prod} one has $\T_k (A_*) \ge \a \T_k (A)$ and as before, we obtain 
    $\dim_{\a,k} (A) \le C' (1-\a^{1/k})^{-2} p \omega^{-1}$ as required.  
    Thus we need to estimate $\sigma_1,\sigma_2$. 
    Using the H\"older inequality, 
    %\eqref{f:T_k_prod} 
    Rudin's Theorem \ref{t:Rudin} and calculations as in \eqref{tmp:02.05_1} (or see \cite[Proposition 5.3, formula (5.11)]{ss_dim}), we get 
%    (or see calcualtions in \cite[Theorem 5.1, after (5.4)]{ss_dim}), we get 
\[
    \sigma^{p}_2 \le \T_p ( \mathcal{E}) \T^{p-1}_{k-1} (A) |A|^{2k-2}
        \le 
            (Cp/l)^p |A|^{2p} \T^{p}_{k-1} (A) \left( |A|^{2k-2} \T^{-1}_{k-1} (A)\right)
\]
    and hence by our choice of the parameters $p$ and $l$, we have 
\[
    \sigma_2 \le \T_k (A) \cdot Cp l^{-1} \omega^{-1} \left( |A|^{2k-2} \T^{-1}_{k-1} (A)\right)^{1/p} \le \T_k (A) (1-\a^{1/k})^2 /100 \,,
\]
    say. Clearly, by the Cauchy--Schwarz inequality 
    $$
        \sigma^2_1 \le 4\sigma_0 \sigma_2 \le \T^2_k (A) (1-\a^{1/k})^2/25
    $$
    hence $\sigma_1 \le \T_k (A) (1-\a^{1/k})/5$ and thus this sum is also negligible.

    \bigskip 
 
\begin{comment}
    Now let us derive the lower bound for $\T_k(A)$ in terms of the dimensions. 
    Put $d=\dim(A)$. 
    The bound $\T_k (A) \ge |A|^{2k-d} (4k)^{-d}$ can be obtained rather quickly. 
    %is rather obvious. 
    Indeed, let $\Lambda =\{ \la_1,\dots,\la_d\} \subseteq A$ be a dissociated set such that any $a\in A$ can be expressed as $a=\sum_{j=1}^d \eps_j \la_j$, where $\eps_j \in \{0,\pm 1\}$.  
    Then for any $a_1,\dots,a_{2k-d} \in A$, we  have
    $$
    a_1+a_2 + \dots +a_k - a_{k-1} - \dots -a_{2k-d} \in \sum_{j=1}^d \eta_j \la_j \,,$$
    where $|\eta_j|< 2k$. 
    Put $\ov{\Lambda} = [-2k+1,2k-1] \cdot \Lambda$.
    Since $\Lambda \subseteq A$, we have by the H\"older inequality that 
    $|A|^{2k-d} \le (4k)^d \T_k (A)$ as required. 
\end{comment}

    It remains to prove \eqref{f:T_k_dim_2}, \eqref{f:T_k_dim_2'} and the argument is almost the same. 
    Let us delete zero from $A$ and with some abuse of the notation we will write $A$ for the remaining set. 
   Recall that for any $l$ we denote by $\ov{A}^{l+1}$ the set of all vectors $\vec{a} = (a_1,\dots, a_{l+1}) \in A^{l+1}$ such that all $a_1,\dots, a_{l+1}$ are different. 
    Clearly, $|\ov{A}^{d+1}|\ge |A|^{d+1} \exp(-d^2 |A|^{-1}) \ge |A|^{d+1} \exp(-d)$. 
    By the definition of the additive dimension for any $\vec{a} \in \ov{A}^{d+1}$ there is $\vec{\eps} \in \{0,\pm 1\}^{d+1}$ such that $\langle \vec{a}, \vec{\eps} \rangle = 0$, see formula \eqref{def:another_def}. 
    Write $\mathrm{wt} (\vec{\eps}) = \sum_{j=1}^{d+1} |\eps_j|$ and since $0\notin A$, it follows that $\mathrm{wt} (\vec{\eps}) \ge 2$. 
    Let $1\le \D \le (d+1)/2$ be a parameter. 
    We have 
\[
    |A|^{d+1} \exp(-d) 
    \le 
    |A|^{d+1} \sum_{w=2}^{d+1} \exp(-w^{2} |A|^{-1}) 
    \le \sum_{w=2}^{d+1}\, \sum_{\vec{\eps} \in \{0,\pm 1\}^{d+1},\, \mathrm{wt} (\vec{\eps})=w} |\{ \vec{a} \in \ov{A}^{d+1} ~:~ \langle \vec{a}, \vec{\eps} \rangle = 0 \}|  
\]
\[
    \le 
    \sum_{w=2}^d\, |A|^{d+1-w} 2^w \binom{d+1}{w} \cdot \T_{[w/2]} (A) |A|^{w-2[w/2]}  
    =
\]
\[
    = |A|^{d+1} \sum_{w\le \Delta} 2^w \binom{d+1}{w} \T_{[w/2]} (A) |A|^{-2[w/2]} 
    +
    |A|^{d+1} \sum_{w>\Delta} 2^w \binom{d+1}{w} \T_{[w/2]} (A) |A|^{-2[w/2]}
    =
\]
\begin{equation}\label{tmp:05.01_1}
    =\sigma_1 + \sigma_2 := \sigma \,.
\end{equation}
    Let us obtain an upper bound for  the sum $\sigma_1$.
    Trivially estimate $\T_{[w/2]} (A)$ as $\T_{[w/2]} (A)\le |A|^{2[w/2]-1}$ (below in the paper we will use some better bounds), we  see that 
\begin{equation}\label{tmp:05.01_1+}
    \sigma_1 \le 2|A|^{d} \left( \frac{2e(d+1)}{\D} \right)^\D
    \le  |A|^{d} \left( \frac{4eML}{\D} \right)^\D \,.
\end{equation}
    We choose $\D$ such that 
\[
    \frac{\D^2}{|A|} + \D \log \left( \frac{8ed}{\D} \right) \ll \D \log \left( \frac{8ed}{\D} \right) \ll L
\]
    or, in other words, we take $\D = c L /\log M$, where $c>0$ is a sufficiently small absolute constant. 
    %Hence 
    It gives us 
    $\sigma_2 \ge 2^{-1} |A|^{d+1} \exp(-d)$.
    %if we take $\D = c L /\log M$, where $c>0$ is a sufficiently small absolute constant. 
    Thus \eqref{tmp:05.01_1} 
    %gives us 
    implies 
\[
    \exp(-d) \le 2 \sum_{w>\Delta} 2^w \binom{d+1}{w} \T_{[w/2]} (A) |A|^{-2[w/2]} 
    \le 
        3^{d+2} \max_{w>\D} \{ \T_{[w/2]} (A) |A|^{-2[w/2]} \}
        =
\]
\begin{equation}\label{tmp:05.01_1++}
        =
        |A|^{C_* M} \max_{w>\D} \{ \T_{[w/2]} (A) |A|^{-2[w/2]} \} \,,
\end{equation}
    where $C_*>0$ is an absolute constant. 
    Thus we have obtained \eqref{f:T_k_dim_2} and to get \eqref{f:T_k_dim_2'} we repeat the calculations from  \eqref{tmp:05.01_1}---\eqref{tmp:05.01_1++} with $A=B$.
    Namely, writing $\T_{d+1} (A) = \frac{|A|^{2d+1}}{Q^{d}}$, $Q \le |A|$ and using the H\"older inequality 
    \begin{equation}\label{tmp:01.02_1}
        \T_l (B) \le \T_l (A) \le \frac{|A|^{2l-1}}{Q^{l-1}}     
    \end{equation} 
    for all $2\le l \le d+1$, we  
    %obtain 
    derive 
%\[
%    1\le 3^d \max_{w>\D} \{ \T_{[w/2]} (B) |B|^{-2[w/2]} \}
%    \le 3^d \max_{w>\D} \{ \T_{[w/2]} (A) |B|^{-2[w/2]} \} \le  3^d \max_{w>\D} \{ |A|^{-1} \d^{-2[w/2]} Q^{1-[w/2]} \} %\,.
%\]
%    In view of bound \eqref{tmp:01.02_1}, as well as  \eqref{tmp:05.01_1+} we see that 
\[
    \sigma \le 
%    |B|^{d+1} |A|^{-1} Q^2 \left( \frac{2e(d+1)}{\D \d \sqrt{Q}} \right)^\D 
    |B|^{d+1} Q|A|^{-1} \sum_{w=2}^{d+1} Q^{-[w/2]} \left( \frac{2e(d+1)}{w \d} \right)^w 
    %= |A|^{d-1} Q^2  \left( \frac{2eML}{\D \sqrt{Q}} \right)^\D 
    \,.
\]
%    it is possible 
%    and hence using our choice of the parameter $\D$ one has 
%\[
%    |A|  Q^{[w/2]-1} \d^{2[w/2]} \le 3^d  
%\]
%    and thus $d\gg L \log (Q\d^2)/\log M$. 
%    By the definition of the quantity $Q$, we finally obtain 
%\[
%    d\gg \left( \frac{\log |A|}{\log M } \cdot \log \left( \frac{|A|^{2d-1}}{\T_d (A)} \right) \right)^{1/2} - \log |A| \cdot \frac{\log 1/\d}{\log M} \,. 
%\]
    and hence automatically $d\gg \d Q^{1/3}$, say. 
    By the definition of the quantity $Q$ we see that 
    $$ d \log (d/\delta) \gg d\log Q = \log (|A|^{2d+1} \T^{-1}_{d+1} (A) )
    $$
    and we finally, obtain the required bound. 
    This completes the proof. 
$\hfill\Box$    
\end{proof}

%\bp 

\begin{corollary}
%    Let $\dim_{1/2} (A) = K \log |A|$. 
    Let $\a\in (0,1]$ be a real number, $k$ be a positive integer and $\dim'_\a (A) := \min_{l\in  [2,k]} \dim_{\a,l} (A)$. 
    Then for any $2\le l\le k$ one has
\[
    \T_l (A) \le 
    \left( \frac{C (\a) k^3}{\dim'_{\a} (A)}\right)^{l}  \cdot |A|^{2l}\,.
%    \left( \frac{16 Ck^2}{K} \right)^{l} \cdot |A|^{2l-1}\,.
\]
\end{corollary}

Thus indeed the dimension of a set $A$ is closely connected with the quantity $\T_k (A)$: see the lower bound for $\dim (A)$ in \eqref{f:T_k_dim_2'} and, on the other hand,  assuming $\dim_{\a,k} (A)\gg \dim (A)$, $\a=2^{-k}$,
 say,  as well as putting $k =c \dim (A)$ (here $c>0$ is a small absolute constant) in \eqref{f:T_k_dim_1-}, we obtain the upper bound 
$$
    \dim (A) \ll \dim_{\a,k} (A) \ll \log (|A|^{2k} \T^{-1}_k (A)) \,.
$$
Also, if, say, $\T_{l} (A) \ll |A|^l$ for a certain fixed number $l$, then we get from \eqref{tmp:05.01_1}---\eqref{tmp:05.01_1++}  
and the trivial estimate $\T_s (A) \le |A|^{2s-2l} \T_l (A)$, $s\ge l$
that $\dim(A) \gg l \log |A|$ and this is a non--trivial bound. 
Finally, one can see that if we have the first inequality in \eqref{f:T_k_dim_2}, that is, 
\[
    \T_m (A) \ge \frac{|A|^{2m}}{2d 4^m e^d \binom{d+1}{2m}} 
\]
for $\log |A|/\log M \ll m \le \dim (A)/2 = d/2$, then an application of Theorem \ref{t:T_k_dim}, formula \eqref{f:T_k_dim_1-}  with $\a = 2^{-m}$ gives us a set $A_* \subseteq A$, $\T_m (A_*) \ge 2^{-m} \T_m (A)$ and $\dim (A_*) \le \exp (O(M \log M)) \cdot  \log |A|$.
%Thus, 
%It gives us 
Thus we obtain 
the second part of Theorem \ref{t:T_k_dim_intr} of the introduction and it 
shows 
one more time that the condition of having large $\T_k$ is roughly equivalent to the condition of having small dimension. 
%and small dimension of a large subset are roughly equivalent. 
%such things  

\bigskip

From \eqref{f:dim_sum*}, it follows that for any $x\in \Gr$ and an arbitrary $A\subseteq \Gr$ one has $\dim (A+x) \lesssim  \dim (A)$. We improve the last bound in Theorem \ref{t:dim(A+x)} below. 

\begin{theorem}
    Let $\Gr$ be an abelian group, 
%    Let $k$ be a positive integer,
    $A, X\subseteq \Gr$ be sets.
    Then
\begin{equation}\label{f:dim(A+X)}
    \dim (A) \ll \dim (A+X) \ll |X| \dim (A) \,.
\end{equation}
    In particular, for any $x\in \Gr$ one has 
\begin{equation}\label{f:dim(A+x)}
    \dim (A) \sim \dim (A+x) \,.
    %= O(\sqrt{\dim_k  (A) k \log (\dim_k (A))}) \,.
\end{equation}
\label{t:dim(A+x)}
\end{theorem}
\begin{proof} 
    We begin  with \eqref{f:dim(A+x)}. 
    Let $d=\dim (A)$. 
    Consider a dissociated set $\Lambda +x  \subseteq A+x$ such that $\dim (A+x) = |\Lambda| = D$.
    Our task is to show that $d\gg D$. 
    %prove $D-d \ll \sqrt{dk \log d}$. 
    If $\Lambda$ is a dissociated set, then $D\le d$ and there is nothing to prove. 
    Thus we can suppose that 
    %the set 
    $\Lambda$ is not a dissociated set but nevertheless, we show that $\Lambda$ is rather close to be dissociated. 
    Indeed, for any positive $k$ we have $\T_k (\Lambda) = \T_k (\Lambda+x)$ and thus by Theorem \ref{t:Rudin} one has $\T_k (\Lambda) \le (Ck)^k |\Lambda|^k$.  
    We now substitute this bound into the proof of Theorem \ref{t:T_k_dim}. 
In the notation of this theorem we have the following restriction on the parameter $\D$
%\begin{equation}\label{tmp:03.05_1}
\[
    \Delta \log (8e d/\D) + \D \log (C\D) \ll \D \log D \,.
\]
%\end{equation}
%where $C_* = 4e$, say. 
Hence it is possible to choose $\D= c D$, where $c>0$ is an absolute (small) constant. 
Using estimate \eqref{f:T_k_dim_2} of Theorem \ref{t:T_k_dim}, we get 
\[
    m \log D - m \log (Cm) \ll \dim(\Lambda) \le d \,,
\]
    Recalling that $m\ge \D$, we obtain the required result.

    Now let us obtain \eqref{f:dim(A+X)}. 
    The first inequality follows from \eqref{f:dim(A+x)} due to 
\[
    \dim (A) \sim \dim (A+x) \le \dim (A+X) \,,
\]
    where $x$ is an arbitrary element of $X$. 
    Further by \eqref{f:dim(A+x)} and by subadditivity of $\dim (\cdot)$ one has 
\[
    \dim (A+X) \le \sum_{x\in X} \dim (A+x) \ll |X| \dim (A)
\]
    %Now let $\Lambda_* \subseteq A+X$ be a dissociated set such that $|\Lambda_*| = D_*  = \dim (A+X)$.
    %%and let $d = \dim (A)$ as before. 
    %By the pigeonhole principle there is $x\in X$ and $S\subseteq A \cap (\Lambda_* - x)$ such that $|S| \ge D/|X|$. 
    %Then $\dim (A) \ge $ As above 
    as required.
%    This completes the proof. 
$\hfill\Box$    
\end{proof}

%\bp 

%\begin{remark}
%    The argument of the proof of Theorem \ref{t:dim(A+x)} allows to estimate a more general sumset $\dim_k (A+H)$, where $H=\{h_1,\dots,h_{|H|}\}$ if we can control the number of different sums $\sum_{j} \eps_j h_j$ with  $|\eps_j|\le k$.  
%\end{remark}

\bp

In the next sections we will use the additive dimensions of a set $A$ to estimate $\T_k (A)$. 
%and hence 
It is well--known that the later quantity can be used to estimate 
the Fourier coefficients of the characteristic function of $A$. Let us make a remark on a simple connection of the Fourier transform of $A$ and $\dim (A)$. 
Of course Proposition \ref{p:dim_Fourier} is non--trivial for sets $A$ with $|A| = o(N)$ only.

\begin{proposition}
    Let $N$ be a prime number, $A\subseteq \Z/N\Z$ and for all $r\neq 0$ one has $|\FF{A}(r)| \le \eps |A|$, $\eps \le 1/4$.
    Then $\dim (A) \gg \log N$. 
\label{p:dim_Fourier}
\end{proposition}
\begin{proof} 
%    One can assume that $|A|\le N/2$ because otherwise the result is trivial. 
    Let $d=\dim (A)$ and $\Lambda = \{\la_1,\dots, \la_d\} \subseteq A$ be a maximal dissociated subset of $A$.
    Suppose that $\dim (A) \le  c\log N$, where $c>0$ is a sufficiently small absolute constant. 
    Then by the Dirichlet Theorem there is $q\neq 0$ such that $\| q\la_j \| \le (N-1)^{1-1/d} < N/8$.
    In other words, $qA$ belongs to the following arithmetic progression $P=(-N/8, N/8)$.
    Consider another arithmetic progression $Q = P+P$, $|Q| < 2|P|$.
    Then for any $a\in P$ one has $r_{Q+Q} (a) \ge |P|$. 
    Hence using the Fourier transform and the Parseval identity, we get 
\[
    |A||P| \le \sum_{a\in qA} r_{Q+Q} (a) < \frac{4|P|^2 |A|}{N} + N^{-1} \sum_{r\neq 0} \FF{A} (q r) \FF{Q} (r)^2 
    < 2^{-1} |A||P| + 2 \eps |A| |P|  \le |A| |P|
\]
    and this is a contradiction. 
%    This completes the proof. 
$\hfill\Box$    
\end{proof}

\section{On the additive dimensions of multiplicative subgroups} 
\label{sec:subgroups}

In this section we consider the case of multiplicatively rich sets, i.e. sets $A\subseteq \mathcal{R}$ with $|AA|\ll |A|$, e.g., multiplicative subgroups.
The property of having small product set 
is rather restrictive and implies that all considered dimensions of Sections \ref{sec:general}, \ref{sec:T_k} are essentially the same for $A$ with $|AA| \ll |A|$.  
In particular, it allows us to estimate $\T_k^{+} (A)$ for such sets $A$.
Also, we give rather good lower bounds for the additive dimensions of multiplicative subgroups in the prime field.

\begin{lemma}
    Let $\mathcal{R}$ be a commutative ring without divisors of zero, $A\subseteq \mathcal{R}$ and $|AA|\le D|A|$.
    Put $d=\dim (A)$.
    Then 
\begin{equation}\label{f:AA_dim}
    \T^{+}_k (A) \le |A|^{2k} \left( \frac{C k D^6 \log ^2 d}{d} \right)^k \,,
\end{equation}
    where $C>0$ is an absolute constant.\\
    Now let $N$ be a prime, and $\mathcal{R} = \R$ or $\Z/N\Z$. Suppose that $S\subseteq A$.
    Then for any $s>0$ one has 
\begin{equation}\label{f:AA_dim_ds} 
    \mathcal{D}_{s,N} (S) \gg \frac{|S|\mathcal{D}_{s,N} (A)}{|A|D^3 \log (D|A|/\mathcal{D}_{s,N} (A))} \,.
\end{equation}
\label{l:AA_dim}
\end{lemma}
\begin{proof}
    Given a set $Z\subseteq \mathcal{R}$ we write $\T_k (Z)$ for $\T^{+}_k (Z)$.
    Let $\Lambda = \{\la_1,\dots, \la_d\} \subseteq A$ be a maximal dissociated subset of $A$.
    Let us apply the standard probability argument (see, e.g., \cite[Exercise 1.1.8]{TV}).
    We take elements of $A\Lambda\Lambda^{-1}$ with probability $p=\frac{C_* \log d}{d}$ (here $C_*>0$ is an appropriate large  constant) uniformly at random and form a set $X$. 
    The probability that a fixed $z\in A\Lambda$ does not belong to $X\Lambda$ is $(1-p)^d$ and hence the expectation of the cardinality of elements of $A$,  which do not in $X\Lambda$ is at most $|A\Lambda|(1-p)^d \le D|A|(1-p)^d$.
    Denoting this set as $\Omega$, we have $|\Omega| \le D|A|(1-p)^d$ and the expectation of size of $X$ is $p|A\Lambda\Lambda^{-1}|\le p |AAA^{-1}| \le D^3 p |A|$ by the Pl\"unnecke inequality \eqref{f:Pl-R}. 
    %, see \cite{TV}. 
    Hence applying Theorem \ref{t:Rudin}, we derive
\[
    \T^{1/2k}_k (A) \le \T^{1/2k}_k (\Omega) + \sum_{x\in X} \T^{1/2k}_k (x\Lambda)  
    \le |\Omega|^{1-1/2k} + p |A| D^3 \sqrt{C'kd} 
    \le 
\]
\[
    \le 
    (D|A|(1-p)^d )^{1-1/2k} + p |A| D^3 \sqrt{C'kd} \,,
\]
    where $C'>0$ is an absolute constant. 
    By our choice of $p$, we get 
\[
    4^{-k} \T_k (A) \le D^{2k-1} |A|^{2k-1} \exp(-C_* (2k-1) \log d)+  |A|^{2k} \left( \frac{C' C^2_* k D^6 \log ^2 d}{d} \right)^k 
    \le
\]
\[
    \le 
    |A|^{2k} \left( \frac{C k D^6 \log ^2 d}{d} \right)^k 
\]
 as required.

 Let us obtain estimate \eqref{f:AA_dim_ds}. We use the same argument replacing $\Lambda$ with $S$ and $d$ with $|S|$. 
 %We will chose 
 Let us choose 
 the parameter $p$ (the probability of our random choice) later.
 With high probability we find two sets $X$ and $\Omega$ such that $|X| \le D^3 p |A|$, $|\Omega| \le D|A| (1-p)^{|S|}$ and $A \subseteq XS \bigsqcup \Omega$. 
 We have 
\[
    \mathcal{D}_{s,N} (A) \le |X| \mathcal{D}_{s,N} (S) + |\Omega| \le 
    |X| \mathcal{D}_{s,N} (S) + D|A| (1-p)^{|S|}
    \le
    2 |X| \mathcal{D}_{s,N} (S) \,,
\]
    where we have chosen $p=\frac{C_* \log (D|A|/\mathcal{D}_{s,N} (A))}{|S|}$ (here $C_*>0$ is an appropriate absolute constant). 
    Recalling that $|X| \le D^3 p |A|$, we obtain the result. 
 This completes the proof. 
$\hfill\Box$   
\end{proof}

\bigskip 

We now obtain an analogue of Lemma \ref{l:AA_d_s} for the dimension $\dim(A)$ and for sets with  $|AA| \le D|A|$.

\begin{corollary}
    Let $N$ be a prime, $A\subset \R$ or $A\subseteq \Z/N\Z$ be a set  with $|AA| \le D|A|$ and $k$ be a positive integer. 
    Put $d =\dim_k (A)$ and 
    suppose that for a certain $s>0$ one has 
    $\mathcal{D}_{s,N} (A) \ge 
    %\_phi(N)^s 
    |A|/T$, $T\ge 1$.
    Then 
\begin{equation}\label{f:AA_d_sl}
    d \gg \frac{s \log (N-1)}{\log (dk^2 D^3 T\log (DT))} \,.
\end{equation}
\label{c:AA_d_sl}
\end{corollary}
\begin{proof} 
    Let $\Lambda = \{\la_1,\dots, \la_d\} \subseteq A$ be a maximal $k$--dissociated subset of $A$.
    %In other words, 
    It means that 
    for any $a\in A$ there is $t\in [k]$ such that $ta= \sum_{j=1}^d l_j \lambda_j$, where $|l_j|\le k$. 
    In other words, $ta \in \Span_k (A)$. 
    By the pigeonhole principle there is $t\in [k]$ and $S\subseteq A$ such that $|S|\ge |A|/k$ and for any $x\in S$ we have $tx \in \Span_k (A)$. 
    Using  Lemma \ref{l:AA_dim}, we see that 
\[
    \mathcal{D}_{s,N} (S) \gg \frac{|S|}{k TD^3 \log (DT)} \,.
\] 
    By the Dirichlet Theorem we find $q \in (\Z/N\Z)\setminus \{0\}$
    %, we have  that 
    such that 
    $\| q\lambda_j/N\| \le (N-1)^{-d^{-1}}$.
    Since $tx \in \Span_k (A)$, it follows that $\| qtx/N\| \le dk(N-1)^{-d^{-1}}$.
    Thus as in Lemma \ref{l:AA_d_s}, we obtain 
\[
    \frac{|S|}{k T D^3  \log (DT)} \ll \sum_{x\in S} \left\| \frac{qtx}{N} \right\|^s \le |S| dk(N-1)^{-sd^{-1}} \,.
\]
 This completes the proof. 
$\hfill\Box$   
\end{proof}

\bigskip 

We now show that in the case of multiplicative subgroups $\G$ of $\F^*_p$  there is almost no difference between $\dim (\G)$ and $\dim_\a (\G)$. 
It allows us to use Theorem \ref{t:T_k_dim} to estimate $\T_k^{+} (\Gamma)$ but, actually, in this specific case Lemma \ref{l:AA_dim} works better.

\begin{lemma}
    Let $\Gamma < \F^*_p$ be a multiplicative subgroup, $k$ be a positive integer, and $\a \in (0,1]$ be a real number. 
%    Put $d=\dim (A)$.
    Then 
\begin{equation}\label{f:subgroup_dim}
    \frac{\a \cdot \dim (\G)}{\log |\G|} \ll \dim_{\a,k} (\G) \le \dim (\G) \,. 
\end{equation}
%    where $C'>0$ is an absolute constant. 
\label{l:subgroup_dim}
\end{lemma}
\begin{proof}
    Let $t=|\G|$ and consider an arbitrary set $B\subseteq \G$ such that $\T_k^{+} (B)\ge \a \T^{+}_k (\G)$. 
    Our task is to estimate $\dim (B)$ from below. 
    We have $\T^{+}_k (\G) = \sum_{x,y \in \G} r_{(k-1)\G - (k-1)\G} (x-y)$. 
    Clearly, the hermitian matrix $M(x,y) = r_{(k-1)\G - (k-1)\G} (x-y)$, $x,y\in \G$ is $\G$--invariant in the sense $M(\g x,\g y) = M(x,y)$ for any $\g\in \G$. 
    Hence the eigenfunctions $f_\a$, $\a \in [t]$ of $M$ are just normalized  characters of the subgroup $\G$ (see \cite[Proposition 3]{s_ineq}). 
    In particular, the main eigenfunction $f_1 (x)$ equals  $\G(x)/t^{1/2}$, where $x$ runs over $\G$ and the correspondent eigenvalue $\mu_1 = \langle M f_1, f_1 \rangle = \T_k^{+} (\G)/t$. 
    It follows that
\[
    \a \T_k^{+} (\G) \le \T_k^{+} (B) \le 
    \sum_{x,y\in B} r_{(k-1)\G - (k-1)\G} (x-y) = \langle M B, B \rangle
    =
    \sum_{\a \in [t]} \mu_\a \langle B, f_\a \rangle^2 \le 
    \T_k^{+} (\G) |B|/t \,, 
\]
    and hence $|B| \ge \a t$. 
    Using the random choice as in Lemma \ref{l:AA_dim}, we find $X\subseteq \G$ such that $|X| \ll \a^{-1} \log t$ and $\G \subseteq XB$.
    But then by the subadditivity of the dimension $\dim (\cdot)$, we get 
    $$
        \dim (\G) \le |X| \dim (B) \ll \a^{-1} \log t \cdot \dim(B)
    $$ 
 as required.  
%    This completes the proof. 
$\hfill\Box$   
\end{proof}

%\section{Applications} 

\bigskip 

We now obtain some applications to the growth of multiplicative subgroups in $\F_p^*$. 

Basis properties of very small subgroups were studied in \cite{Bourgain_coset,K,KS}. In a natural way the authors of these papers were interested in obtaining upper bounds for exponential sums over such subgroups 
%and 
but 
in our approach we do not want to use this machinery. 
Nevertheless, both methods rest on lower bounds for the quantity $\mathcal{D}_{2,p}$, see Lemmas  \ref{l:KS_subgroups_B}, \ref{l:S_subgroups_B}  below. 
%We need Theorem 4.2 of \cite{KS}.
%See the proof of Theorem 4.2 of \cite{KS}. 
We start with Theorem 4.2 of \cite{KS}. 

\begin{lemma}
    Let $g\in \F^*_p$ has the multiplicative order equals $t$.
%    be a multiplicative subgroup and $t=|\G|$.
    Then for any $2\le r\le \_phi(t)$
    one has 
\[
    p^2 \cdot \mathcal{D}_{2,p} (\{1,g,\dots, g^{t-1}\})
    \ge  \left( \frac{p^{2(r-1)}t}{r \gamma^{r-1}_{r-1}} \right)^{1/r} \,,
\]
    where $\gamma_{r-1}$ is $(r-1)$th Hermite constant, $\gamma_{r-1} = O(r)$.
\label{l:KS_subgroups_B}
\end{lemma}

Also, we need \cite[Theorem 1, Lemma 6]{K}.

\begin{lemma}
    Let $g\in \F^*_p$ be a primitive root,  $\eps \in (0,1)$,  $k\ge k(\eps)$ be a sufficiently large positive integer, $\beta = g^k$, $m=[6\ln p (\ln \ln p)^4]$, and 
\begin{equation}\label{cond:S_subgroups_B}
    p\ge \frac{k\ln k}{\ln^{1-\eps} (\ln k+1)} \,.
\end{equation}
%    be a multiplicative subgroup and $t=|\G|$.
    Then 
\begin{equation}\label{f:S_subgroups_B_1}
    \mathcal{D}_{2,p} (\{\beta,\dots, \beta^m\}) \ge \frac{1}{(\ln p)^{3\eps/4}} 
    %\,,
    \,.
\end{equation}
%    and for any $k\ge 2$ one has 
%\begin{equation}\label{f:S_subgroups_B_2}
%    \max_{a\neq 0} \left|\sum_{x\in \F_p} e^{2\pi i ax^k/p} \right| 
%    \le p(1 - c_0/(\ln k)^{1+\eps}) \,,
%\end{equation}
%    where $c_0>0$ is an absolute constant. 
\label{l:S_subgroups_B}
\end{lemma}

Using the results above we can 
%estimate
obtain a good lower bound for $\dim (\G)$ and $\dim_k (\G)$.

\begin{corollary}
    Let $\G<\F_p^*$ be a multiplicative subgroup, $t=|\G|$.
    Then 
\begin{equation}\label{f:G_dim_B_1}
    \dim (\G) \gg \min \left\{\frac{\log p}{\log \log p}, \frac{\log p}{\log t}, \_phi (t) \right\} \,.
\end{equation}
    In particular, if $t\gg \frac{\log p}{\log^{1-\eps} \log p}$ for a certain $\eps \in (0,1)$, then 
\begin{equation}\label{f:G_dim_B_2}
     \dim (\G) \gg_\eps \frac{\log p}{\log t}\,.
\end{equation} 
    Similarly, for any positive integer $k$ one has 
\begin{equation}\label{f:G_dim_B_1k}
    \dim_k (\G) \gg \min \left\{\frac{\log p}{\log \log p}, \frac{\log p}{\log t}, \frac{\log p}{\log k}, \_phi (t) \right\} \,.
\end{equation}
\label{c:G_dim_B}
\end{corollary}
\begin{proof}
    Let $d=\dim (\G)$ and $d_k = \dim_k (\G)$. 
    We start with \eqref{f:G_dim_B_1}. 
    Applying Lemma \ref{l:KS_subgroups_B} and Lemma \ref{l:AA_d_s} with $T=C p^{2/r} t^{1-1/r} r$, where $C>0$ is an appropriate constant and $r=\_phi(t)$,  we see that 
\[
    d \gg_\eps \log p \cdot \min \{ r/\log p, \log^{-1} t, \log^{-1} d \} 
\]
 as required.  
 To obtain \eqref{f:G_dim_B_2} one can use both Lemmas \ref{l:KS_subgroups_B}, \ref{l:S_subgroups_B} 
 %%(or just inequality \eqref{f:G_dim_B_1} 
 %Now 
 and we prefer to 
 %we 
 %use 
 apply 
 inequality \eqref{f:S_subgroups_B_1} of Lemma \ref{l:S_subgroups_B} with $k=(p-1)/t$ and  $T=m (\ln p)^{3\eps/4}$.
 Here we have splitted the sequence $\{ g^{j(p-1)/t} \}_{j=1}^t$ onto subsequences of length $m$ and also we have assumed 
 %upposing 
 that $t\ge m$.  One has 
\[
    d \gg \min \left\{ \frac{\log p}{\log \log p}, \frac{\log p}{\log m} \right\}
    \gg \frac{\log p}{\log t} \,.
\]
    Also, by the assumption of Lemma \ref{l:S_subgroups_B} we need to check that $k\gg_\eps 1$ but if not, then $t\gg_\eps p$ and bound \eqref{f:G_dim_B_2} is trivial. 
    %Finally, 
    Similarly, 
    if $t<m$, then by the average arguments 
    %$T=t^2 m^{-1} (\ln p)^{3\eps/4}$ 
    $T\sim m (\ln p)^{3\eps/4}$, $\log m \sim \log t$
    and again estimate \eqref{f:G_dim_B_2} follows.

    To obtain \eqref{f:G_dim_B_1k}, we get by Corollary \ref{c:AA_d_sl} with $D=1$
    and our choice of the parameter $T=C p^{2/r} t^{1-1/r} r$, $r=\_phi (t)$ that 
\[
    d_k \gg \log p \cdot \min \{ r/\log p, \log^{-1} t, \log^{-1} d_k, \log^{-1} k \}
\]
    as required.    
%    This completes the proof. 
$\hfill\Box$   
\end{proof}

\bigskip

Using the obtained lower bounds for the dimensions from Corollary  \ref{c:G_dim_B}, we  derive rather good upper bounds for the quantity $\T^{+}_k (\G)$ in the case of  small subgroups $\G$.

\begin{corollary}
    Let $\G<\F_p^*$ be a multiplicative subgroup.
    %, $t=|\G|$.
    If $|\G| \le \frac{\log p}{\log \log p}$, then for any $k\ge 2$ one has 
\begin{equation}\label{f:growth_bound1}
    \T^{+}_k (\G) \le |\G|^{k} 
    %\left( \frac{|\G|}{C_* k \log^2 |\G|\cdot \log \log |\G|} \right)^k \,,
    (C_* k \log^2 |\G|\cdot \log \log |\G|)^k \,,
\end{equation}
    where $C_*>0$ is an absolute constant. 
    If $|\G| \ge \log p$, then for any $k\ge 2$ one has 
\begin{equation}\label{f:growth_bound2}
    \T^{+}_k (\G) \le |\G|^{2k} 
    %\left( \frac{\log p}{C_* k \log |\G| \cdot \log^2 (\log_{|\G|} p)} \right)^k \,.
    \left( \frac{C_* k \log |\G| \cdot \log^2 (\log_{|\G|} p)}{\log p} \right)^k \,,
\end{equation}
    and if $\frac{\log p}{\log \log p} \le |\G| \le  \log p$, then 
\begin{equation}\label{f:growth_bound2.5}
    \T^{+}_k (\G) \le |\G|^{2k} \left( \frac{C_* k \cdot \log^2 (\log p)}{\min\{ \_phi (t), \frac{\log p}{\log \log p} \} } \right)^k \,. 
\end{equation} 
    In particular, if $|\G| \sim \log p$, then 
\begin{equation}\label{f:growth_bound3}
    |k\G| = \Omega \left( \left( \frac{|\G|}{k \log^3 |\G|} \right)^k \right) \,.
\end{equation}
\label{c:G_growth}
\end{corollary}
\begin{proof}
    Let $d=\dim (\G)$ and $t=|\G|$. 
    Everything follows from Corollary \ref{c:G_dim_B}. 
    Indeed, if $t \le \frac{\log p}{\log \log p}$, then the minimum in \eqref{f:G_dim_B_1} is attained at $\_phi (t)$ and hence by Lemma \ref{l:AA_dim} and the Cauchy--Schwarz inequality, we get 
%\[
%    |k\G| \ge  \left( \frac{d}{C_* k \log ^2 d} \right)^k 
%    \ge 
%    \left( \frac{\_phi(|\G|)}{C_* k \log ^2 |\G|} \right)^k 
%    \ge 
%    \left( \frac{|\G|}{C_* k \log ^2 |\G| \cdot \log \log |\G|} \right)^k \,.
%\]
\[
    \T^{+}_k (\G) \le t^{2k}  \left( \frac{d}{C_* k \log ^2 d} \right)^{-k} 
    \le 
    t^{2k}
    \left( \frac{\_phi(t)}{C_* k \log ^2 t} \right)^{-k} 
    \le
    t^{2k}
    \left( \frac{t}{C_* k \log ^2 t \cdot \log \log t} \right)^{-k} \,.
\]
    Similarly, if $t \ge \log p$, then the minimum in \eqref{f:G_dim_B_1} is attained at $\frac{\log p}{\log t}$ and estimate \eqref{f:growth_bound2} follows. 
    Finally, if $\frac{\log p}{\log \log p} \le |\G| \le  \log p$, then  $\dim (\G) \gg \min\{ \_phi (t), \frac{\log p}{\log \log p}\}$ and we obtain \eqref{f:growth_bound2.5}. 
% as required.  
    This completes the proof. 
$\hfill\Box$   
\end{proof}

\bp

An alternative method to obtain 
%the result above 
bound  \eqref{f:growth_bound3} 
is to use estimate \eqref{f:A+B_dim} or formulae \eqref{f:dL_nA},  \eqref{f:dL_nA+} of Lemma \ref{l:dL}. 
%We left the intermediate case $\frac{\log p}{\log \log p}<|\G| < \log p$ for the interested reader (just use the estimate $\dim (\G) \gg \min\{ \_phi (t), \frac{\log p}{\log \log p}\}$). 

Finally, we obtain Theorem \ref{t:p^c_intr} from the introduction,  exploiting the stronger fact that there is a good lower bound for $\dim_k (\G)$ for rather large $k$.

\begin{corollary}
    Let $\G<\F_p^*$ be a multiplicative subgroup.  
    Suppose that $\_phi (|\G|) \log |\G| \ge \log p$, $|\G| \le (\log p)^C$, where $C\ge 1$  is an absolute constant. 
    Then there is $n = O(\log^2 p/ \log \log p)$ such that $|n\G| \ge p^{\Omega(1/C)}$.\\
    %for a certain $c=c(C)>0$.  
    Further if $|\G|\le \log p$, then for $n = O(\_phi^2 (t) \log t)$ one has 
    $|n\G| \ge \exp(\log t \cdot \Omega(\min\{ \_phi (t), \frac{\log p}{\log \log p}\}))$.
\label{c:p^c}
\end{corollary}
\begin{proof} 
    Let $t=|\G|$ and take $k=t \log t \ge \dim (\G) \log \dim (\G)$. 
    Then by Corollary \ref{c:G_dim_B} and our assumption $\_phi (|\G|) \log |\G| \ge \log p$ we see that $\dim_k (\G) \gg \log p/\log t$. 
    On the other hand, 
%    using estimate \eqref{f:LY_cor} of Lemma \ref{l:LY}, we see that 
%\begin{equation}\label{tmp:12.04_1}
%    \log p \gg \dim_k (\G) \log \dim_k (\G) \gg \dim_k (\G) \log t \,,
%\end{equation}
    %and thus 
    clearly, 
    $\dim_k (\G) \ll \log p/\log t$. 
    Applying formula \eqref{f:LY_2_particular} of Lemma \ref{l:LY}, we obtain 
\[
    |n\G| \ge \exp (\Omega (\log p/\log t \cdot \log (\log p/\log t) - \log n) ) 
    \ge 
    p^{\Omega(1/C)} \,,
    %\,,
\]
    where  $n = O(\log^2 p/ \log \log p)$. 
%    for a certain $c=c(C)>0$. 
% as required.
    The second part of Corollary \ref{c:p^c} can be obtained in a similar way, just notice that $\dim_k (\G) \gg \min\{ \_phi (t), \frac{\log p}{\log \log p}\}$. 
    This completes the proof. 
$\hfill\Box$   
\end{proof}

\section{Dimensions and the sum--product phenomenon}
\label{sec:sum-product}

We begin this section with estimating multiplicative dimensions of  the difference sets $A-A$ for sets $A \subseteq \mathcal{R}$ such that $|A+A| \ll |A|$. Our new inclusion \eqref{f:D_dim} is interesting in its own right.

\begin{theorem}
    Let $k$ be a positive integer and $A$ be a finite subset of  an abelian ring $\mathcal{R}$ such that $|A+A|\le K|A|$.
    Put $D=A-A$.
    Then 
\begin{equation}\label{f:D_dim}
    [n]/[n] \subseteq D/D \,, \quad \quad \mbox{ where }
    \quad \quad 
    n = \exp (\Omega(\log |A|/\log K)) \,.
\end{equation}
    In particular, for $A\subset \R$ the following holds 
\begin{equation}\label{f:D_dim_2}
    \dim^\times_k (D) \ge \exp (\Omega(\log |A|/\log K)) \,.
\end{equation}
    Hence for any $m\ll \log |A|$ and $A\subset \R$ one has 
\begin{equation}\label{f:D_dim_3}
    |D^m| \ge |D| \exp(m (\Omega(\log |A|/\log K) - \log \log |A|)) %\,,
    \,,
\end{equation}
    as well as
\begin{equation}\label{f:D_dim_3+}
    |D^n| \ge \exp( \exp (\Omega(\log |A|/\log K) )) 
    %\,,
%    \quad \quad 
%        \mbox{ and }
%    |D^{n^2 \log n}| \ge     
%    \quad \quad 
    \,. 
\end{equation}
    Now for $A \subseteq \F_p$ the following holds 
\begin{equation}\label{f:D_dim_3.5}
    \dim^\times_k (D) \gg \min \left\{ n, \frac{\log p}{k^2 (\log \log p)^3} \right\} \,,
\end{equation} 
and for 
an arbitrary $m$ 
%the following holds 
one has 
\begin{equation}\label{f:D_dim_4}
    %\quad \quad 
    %    \mbox{ and }
    %\quad \quad 
    |(D/D)^m| \ge p^{-o(1)} \min \{ n^m, p \} \,.
\end{equation}
\label{t:D_dim}
\end{theorem}
\begin{proof} 
    For any $\lambda_1, \lambda_2 \in \Z\setminus \{0\}$, we have
    in view of Theorem \ref{t:bukh} 
\[
    |\{ (a_1,a_2,a_3,a_4) \in A^4 ~:~ \la_1 (a_1-a_2)= \la_2  (a_3-a_4) \}| = \E (\la_1 \cdot A,\la_2 \cdot A) \ge \frac{|A|^4}{|\la_1 \cdot A+\la_2 \cdot A|} 
    \ge
\]
\[
    \ge |A|^3 \exp(-O(\log K \cdot \log (1+|\lambda_1|)(1+|\la_2|))) > |A|^2 
\]
    for any $\la_1,\la_2 \in [n]$, where 
    $n = \exp (\Omega(\log |A|/\log K))$.   
    Hence $\la_2/\la_1$ can be expressed as $\frac{a_1-a_2}{a_3-a_4}$ or, in other words, $\la_2/\la_1$ belongs to $D/D$. 
    It gives us inclusion \eqref{f:D_dim}.

    Now there are $\Omega(n/\log n)$ primes in $[n]$ and thus in view of \eqref{f:dim_sum*} or inequality  \eqref{f:dim_sum1} of Lemma \ref{l:dim_sum}, we get 
\[
    \dim^\times (D) \gg \frac{n}{\log^2 n}
    \gg 
    \exp (\Omega(\log |A|/\log K)) \,.
\]
    To obtain \eqref{f:D_dim_3} and \eqref{f:D_dim_3+} it remains to use estimates \eqref{f:dL_nA}, \eqref{f:dL_nA+} of Lemma \ref{l:dL}. 
    To get \eqref{f:D_dim_3.5} we take $l\le n$ such that $l^{kl} < p/2$. Then all primes up to $l$ form a $k$--dissociated set $\Lambda$ modulo $p$.
    The condition $l^{kl}<p/2$ is equivalent to  $l \ll \log p/(k\log \log p)$ and by the prime number theorem we have $|\Lambda| \gg \log p/k(\log \log p)^2$.
    Hence $\dim^\times_k (D/D) \ge |\Lambda|$ and it remains to apply estimate \eqref{f:dim_sum1} of Lemma \ref{l:dim_sum} with $l=k$ and $k=2$. 
    Notice that we write $n$ in \eqref{f:D_dim_3.5} but not $n/\log n$ because we can just decrease the constant in the symbol  $\Omega$ in the definition of the number  $n$ from \eqref{f:D_dim}. 
    %it costs  just 

    Finally, to get \eqref{f:D_dim_4} just use a trivial bound 
    $|(D/D)^m| \ge |[n]^m|$ and apply the standard calculations with the divisor function.  
    This completes the proof. 
$\hfill\Box$   
\end{proof}

\begin{remark} 
    If one 
    %replaces 
    switches the operations in the result about, then it is easy to obtain that $[n]/[n] \subseteq L/L$, where $L=\log (A/A)$ for any $A\subset \R$, say, and hence $\dim^\times (L) \ge \exp(\Omega (\log |A|/\log K))$ for an arbitrary $A$ with $|AA|\le K|A|$. 
\end{remark}

The dependence on $m$ in \eqref{f:D_dim_3} is not logarithmic as in \cite{Bradshaw}, \cite{BHR} or \cite{HR-NR}  and the whole bound is much better than \cite[Theorem 2]{s_GCD}. 
%is better than in \cite[]{BHR}.
As in \cite{s_GCD} bound \eqref{f:D_dim_3}  is a step towards the main conjecture from  \cite{BR-NZ}, where authors do not assume that the additional condition of the doubling constant takes place. 
We now obtain a result in $\F_p$ of the same spirit.  
%similar result in $\F_p$. 
It is a byproduct of inclusion \eqref{f:D_dim}.

\begin{corollary}
    Let $p$ be a prime number, $\d\in (0,1)$ be a real number, and $\G \le \F_p^*$ be a multiplicative subgroup, $|\G| \le p^{1-\delta}$. 
    Suppose that $A-A \subseteq \G$ and $|A+A|\le K|A|$. 
    Then 
\begin{equation}\label{f:A-A_in_G}
    |A| \ll \exp (C \log K \cdot \sqrt{\d^{-1} \log (1/\d) \log p}) \,,
\end{equation}
    where $C>0$ is an absolute constant. 
%    depends on $\delta$ only. 
\label{c:A-A_in_G}
\end{corollary}
\begin{proof} 
    Let $D=A-A \subseteq \G$. 
    Applying \eqref{f:D_dim} of Theorem \ref{t:D_dim}, we find $n = \exp (\Omega(\log |A|/\log K))$ such that 
    \[
        [n] \subset [n]/[n] \subseteq D/D \subseteq \G \,.
    \]
    By \cite[Proposition 7]{s_ineq} if 
    $P = a\cdot [k] \subseteq \G$ for an arbitrary   $a\neq 0$, then 
\[
    |P| \le \exp (C \sqrt{\d^{-1} \log (1/\d) \log p}) \,,
\]
    where $C>0$ is an absolute constant. 
    Putting $a=1$ and $k=n$, we obtain the required estimate. 
% as required.  
    This completes the proof. 
$\hfill\Box$   
\end{proof}

\bp

We now obtain the first result on dimensions of sets with small sumset/difference set. In view of forthcoming  Theorem \ref{t:sum-prod_dim} these rather simple bounds \eqref{f:appl_sp_dim_1}, \eqref{f:appl_sp_dim_2} are not so weak.

\begin{proposition}
    Let $A \subset \R$ be a set such that $|A+A| \le K|A|$ or $|AA| \le K|A|$. 
    Then 
\begin{equation}\label{f:appl_sp_dim_1}
    \dim^{\times} (A),~ \dim^{+} (A) \gg \log |A| \cdot \log \left( \frac{\log |A|}{\log K} \right) \,,
\end{equation}
    respectively.
    Now if $A\subset \Z$ and $|AA|\le K|A|$, then 
\begin{equation}\label{f:appl_sp_dim_2}
    \dim^{+} (A) \gg \frac{\log^2 |A|}{\log K} \cdot \log \left( \frac{\log |A|}{\log K} \right) \,. 
\end{equation}
\label{p:appl_sp_dim}
\end{proposition}
\begin{proof} 
Let $L=\log |A|$ and $d=\dim (A)$. 
If $d \gg L \log (L/\log K)$, then there is nothing to prove. 
To obtain \eqref{f:appl_sp_dim_1} we apply  \cite[Theorem 1.4]{BHR}  (in the case $|AA| \le K|A|$ one can alternatively use \cite[Theorem 5]{s_Bourgain}). 
%Say, in 
For example, consider  the case $|AA|\le K|A|$. 
Thanks to 
%\cite[Theorem 5]{s_Bourgain} 
\cite[Theorem 1.4]{BHR} with $f(x) = \exp(x)$, 
we have for all $k$ 
\[
    \T^{+}_{2^k} (A)\ll K^{3\cdot 2^{k+1}} |A|^{2^{k+1}-k+O(1)} \,.
\]
In other words, 
$
%\begin{equation}\label{tmp:T_k_int}
    \T^{+}_k (A) \ll K^{6k} |A|^{2k-\log k + O(1)} \,.
%\end{equation}
$
We now substitute this bound in  the proof of Theorem \ref{t:T_k_dim}. 
In the notation of this theorem we have the following restriction on the parameter $\D$
\begin{equation}\label{tmp:03.05_1}
    \Delta \log (8eKd/\D) \ll L \log \D \,.
\end{equation}
Hence it is possible to choose $\D= c \frac{L}{\log K} \log \left( \frac{L}{\log K} \right)$, where $c>0$ is an absolute small constant (recall that $d\ll L\log (L/\log K)$). 
Using estimate \eqref{f:T_k_dim_2} of Theorem \ref{t:T_k_dim}, we get 
\[
    L \log m \ll d + m \log K \,.
\]
%Taking $k$ such that $2^k \sim \frac{L}{\log K} \log \left( \frac{L}{\log K} \right)$, 
Recalling that $m\ge \D$ and using the bound $\log m \gg \log (L/\log K)$, we obtain the required result.

In a similar way, applying 
%\cite[Theorem 1.3]{ZP}, 
Theorem \ref{t:ZP}, 
we get for any $\eps \in (0,1)$ and $k\ge 2$ 
\begin{equation}\label{tmp:03.05_2}
    \T^{+}_k (A) \le 10^k \beta^{k/\eps}(A) |A|^{k+2\eps k \log k} \,.
\end{equation} 
Clearly, by the Pl\"unnecke inequality \eqref{f:Pl-R} one has $\beta (A) \le K^3$.  
Inserting bound  \eqref{tmp:03.05_2} into the proof of Theorem \ref{t:T_k_dim}, we need to estimate the parameter $\D$ 
%as in 
similar to 
\eqref{tmp:03.05_1}.
Choosing $\eps \sim 1/\D \log \D$, we obtain 
\begin{equation}\label{tmp:03.05_1+}
    \D^2 \log \D \cdot \log K + \D \log (8ed/\D) \ll \D L 
\end{equation} 
and hence it is possible to choose $\D = c \frac{L}{\log K} \cdot \log \frac{L}{\log K}$, where $c>0$ is an absolute small constant (as above we can assume that $d$ is sufficiently small because otherwise there is nothing to prove).  
Applying bound \eqref{f:T_k_dim_2} of Theorem \ref{t:T_k_dim} again, we derive 
\[
    m L \ll d + m \eps^{-1} \log K \ll d + m^2 \log m \log K \,.
\]
    Using the fact that $m\ge \D$, we get $d\gg \frac{\log^2 |A|}{\log K} \cdot \log \left( \frac{\log |A|}{\log K} \right)$. 
%  as required.  
    This completes the proof. 
$\hfill\Box$   
\end{proof}

\bp

%We want 
Our next step is 
to obtain a bound similar to \eqref{def:2^k_beta}  for the energy  $\T_{2^k-1} (A)$. 
It is well--known by the Balog--Szemer\'edi--Gowers Theorem that the property 
%that 
a set $A$ has small sumset correlates with the largeness of its additive energy.
%If $A$ has small 
We show that the smallness of $\beta (A)$ is connected with 
%lower bounds for 
the fact that the energy 
$\T_k (A)$ is large.

\begin{theorem}
    Let $A\subseteq \Gr$ be a set, $k$ be a positive integer and $K\ge 1$ be a real parameter.
    %%$k\ll \log |A|/\log (kK)$.
%%    $k\le \exp \left( \frac{\log |A|}{\log \log |A|} \right)$. 
    Suppose that  $\T_{k} (A) = |A|^{2k-1} K^{1-k}$.
    Then 
    %Then for any positive integer $k$ either $\T_{2^k} (A) \le |A|^{2^{k+1}-1} K^{-k}$ or 
    there is $A_* \subseteq A$ such that $|A_*|\ge |A|/M$,  $\beta (A_*) \le M$, where 
\begin{equation}\label{f:beta_T}
    M = \exp ( C k \log kK/\log k  ) \,,
\end{equation} 
    and $C>0$ is an absolute constant. 
\label{t:beta_T}
\end{theorem}
\begin{proof} 
    Write $\T_j$ for $\T_j (A)$. 
    We have $\T_{k} = |A|^{2k-1} K^{1-k}$ and hence there is $j\in [k]$ such that $\T_{j}  \ge |A|^{2} \T_{j-1}/K$.
    We take the largest $j$ with this property. 
    In particular, 
    \begin{equation}\label{tmp:T_j-1_new}
        |A|^{2k-1} K^{1-k} = \T_k \le (|A|^2/K)^{k-j} \T_j \le  
        (|A|^2/K)^{k-j} |A|^2 \T_{j-1} \,.
    \end{equation} 
    Putting $L= \log (8K |A|^{2j-3} \T^{-1}_{j-1})$ and using the last formula, we see that $L\le \log (8K^{j}) \le 16 k \log K$.
    Now 	by the dyadic Dirichlet principle and the H\"older inequality there is a number $\D >0$ and a set $P = \{ x \in \Gr ~:~ \D <  r_{(j-1)A} (x) \le 2 \D \}$, such that
	\begin{equation}\label{tmp:19.09_3_new}
	2L^2 \D^2 \E(P,A) \ge 
	\T_{j}  \ge \frac{|A|^{2} \T_{j-1}}{K} \ge  	 
	\frac{|A|^2 \D^2 |P|}{K} \,,
	\end{equation}
	and hence $\E(P,A) \ge |P||A|^2/(2L^2 K) := |P| |A|^2/K_*$. 
    We apply Lemma \ref{l:As_BszG} from the Appendix  with $A=P$, $B=A$ and $K=K_*$. According to this Lemma, for any 
    %%$l\le \log (\log |A|/\log K_*)$ 
%%    $l\le \frac{\log |A|}{\log \log |A|}$ 
    $l$ 
    we find a set 
    $H \subseteq \Gr$ such that $|H+H| \ll M^C |H|$ and for a certain $x\in \Gr$ one has $|A\cap (H+x)| \gg |B|/M^C$.
    Here $C>0$ is an absolute constant and $M = (|P|/|A|)^{1/l} K_*^{2^l/l}$.
    Put $A_* = A \cap (H+x)$.
    Then by the Pl\"unnecke inequality \eqref{f:Pl-R}  and definition \eqref{def:beta} of the quantity $\beta$ one has  
\[
    \beta (A_*) \le \frac{|A_*+H+H|}{|H|} \le  \frac{|3H|}{|H|} \ll
    M^{3C} \,.
\]
    It remains to choose $l$ and obtain a good upper bound for $M$. 
    Using \eqref{tmp:19.09_3_new} and \eqref{tmp:T_j-1_new}, we see that 
\[
    2L^2 |A|^{2j}
    \ge 
    2L^2 |A|^2 (\D |P|)^2 \ge  \T_{j}|P|  \ge \frac{|A|^{2} \T_{j-1} |P|}{K}
    \ge 
    |A|^{2j-1} K^{-j} |P|
\]
and hence $|P| \le 2L^2 K^j |A|$. 
It follows that 
\[
    M\le (2L^2 K^k)^{1/l} (2L^2 K)^{2^l/l}
\]
    Recalling that $L \le 16 k \log K$ and choosing $l$ optimally as $l = \log k$, we obtain 
\[
    M \le \exp ( C_* k (\log kK/\log k) ) \,,
\]
    where $C_* >0$ is an absolute constant. 
%%    Now we must have $l\le \frac{\log |A|}{\log \log |A|}$ and hence $k\le \exp \left( \frac{\log |A|}{\log \log |A|} \right)$. 
% as required.  
    This completes the proof. 
$\hfill\Box$   
\end{proof}

\bp 

%The result 
Theorem 
above implies the first (asymmetric) decomposition result.

\begin{theorem}
    Let $A \subset \Z$ be a set, and  $q,s$ be some integer parameters, $s q \log q \ll \log |A|$. %%$\log s \gg q \log q$.
%%    $s\le \exp \left( \frac{\log |A|}{2\log \log |A|} \right)$. 
    Then there exist pairwise disjoint sets $B$ and $C$ such that $A=B\bigsqcup C$ and 
\begin{equation}\label{f:dec_Tk}
    \T^{+}_q (B) \le |B|^{8q/5} 
    \quad \quad 
        \mbox{ and } 
    \quad \quad 
    \T^{\times}_s (C) \le |A|^{2s - \frac{c_* \log s}{q \log q}} \,.
\end{equation}    
    %where 
    Here $c_*>0$ is an absolute constant.
\label{t:dec_Tk}
\end{theorem}
\begin{proof}
    Let $L=\log |A|$, and $\beta = \beta (A)$.
    Choose a number $K$ such that $|A|^{2s - c' (q \log q)^{-1} \log s} := |A|^{2s-1} K^{1-s}$, where the constant $c'>0$ is sufficiently small. 
    %will be chosen later. 
    Our proof is a sort of an algorithm. 
    We construct a decreasing sequence of sets $A=C_1 \supseteq C_2 \supseteq \dots \supseteq C_k$ and an increasing sequence of sets $\emptyset = B_0 \subseteq B_1 \subseteq \dots \subseteq B_{k-1} \subseteq A$ such that for any $j\in [k]$ the sets $C_j$ and $B_{j-1}$ are disjoint and moreover, $A=C_j \bigsqcup B_{j-1}$. 
    If at some step $j$ we have $\T^{\times}_s (C_j) \le |A|^{2s-1} K^{1-s}$ we stop and set $C=C_j$, $B=B_{j-1}$, and  $k=j-1$. 
    Else, we have $\T^{\times}_s (C_j) > |A|^{2s-1} K^{1-s}$. 
    In particular, $|C_j|\ge |A| K^{\frac{1-s}{2s-1}} \ge |A| K^{-1/2}$.
    We apply Theorem \ref{t:beta_T} to the set $C_j$, finding $D_j \subseteq C_j$ such that 
    $|D_j| \ge |C_j|/M$, $\beta:= \beta (D_j) \le M$ and $M$ is given by formula \eqref{f:beta_T}, that is,  
$
    M = \exp ( O(s \log sK/\log s ) ) \,.
$
%    Let $L=\log |A|$ and $\beta = \beta (A)$.
%    We can assume that $M\le |C_j|^{1/2}$, say. 
    We will assume that $|D_j|\ge |A|^{1/2}$, say. 
Using Theorem \ref{t:ZP} with $A=D_j$ and $\eps \sim 1/(q\log q)$, we obtain \[
    \T^{+}_q (D_j) \le 10^q \beta^{q/\eps} |D_j|^{q+2\eps q \log q} 
    \le \beta^{100 q^2 \log q} |D_j|^{q + q/10} \le |D_j|^{q+q/5} \,,
\]
    provided 
\begin{equation}\label{cond:qs}
    q \log q \cdot \log \beta \ll q \log q \cdot s \log sK/\log s \ll L \,.
\end{equation}
    After that we put $C_{j+1} = C_j \setminus D_j$, $B_j = B_{j-1} \bigsqcup D_j$ and repeat the procedure.
    Clearly, our algorithm stops after at most  $K^{1/2} M$ number of steps. 
    Also, it is easy to see that the second estimate in \eqref{f:dec_Tk} holds with $c_* = c'/2$, say.
    It remains to check the first inequality from \eqref{f:dec_Tk}. 
    From the norm property of the energies $\T^{+}_q$
    and our condition \eqref{cond:qs} 
    one has 
\[
    \T^{+}_q (B) \le \left( \sum_{j=1}^k |D_j|^{3/5} \right)^{2q}
    \le 
    |B|^{q+q/5} \cdot (M\sqrt{K})^{2q}
    \le 
    |B|^{8q/5} \,.
\]
    Similarly, condition \eqref{cond:qs}  
    %one can check that 
    gives us $M\sqrt{K} \le |A|^{1/4}$ and hence $|D_j|\ge \sqrt{|A|}$ as required. 
    Finally, by the choice of the quantity $K$ one has 
$$
    s \log K = \log K^s  = c' L (q \log q)^{-1} \log s = 2c_*  L (q \log q)^{-1} \log s
$$
and
%we can assume that 
thanks to our assumption 
$$
    s q \log q  
    %\le c_* L 
    \ll L 
    %\,.
$$
%because 
%Thus 
we see that condition \eqref{cond:qs} satisfies. 
% as required.  
%%    Also, the condition $s\le \exp \left( \frac{\log |A|}{2\log \log |A|} \right) \le \exp \left( \frac{\log |C_k|}{2\log \log |C_k|} \right)$  guarantees that Theorem \ref{t:beta_T} can be applied at each step of our algorithm.  
   This completes the proof. 
$\hfill\Box$   
\end{proof}

\bp

We now obtain our new decomposition result in the spirit of \cite[Corollary 1.3]{Mudgal} and \cite{BW}.

\begin{corollary}
    Let $A \subset \Z$ be a set and $s$ be an integer parameter,  
    \begin{equation}\label{cond:dec_Tk_c}
    s \ll \frac{\log |A|}{\sqrt{\log \log |A| \cdot \log \log \log |A|}} \,.
    \end{equation} 
    Then there exist pairwise disjoint sets $B$ and $C$ such that $A=B\bigsqcup C$ and 
\begin{equation}\label{f:dec_Tk_c}
    \max\{ \T^{+}_s (B), \T^{\times}_s (C) \} \le |A|^{2s-\frac{c_* \sqrt{\log s}}{\sqrt{\log \log s}}} \,,
\end{equation}    
    where $c_*>0$ is an absolute constant.
\label{c:dec_Tk}
\end{corollary}
\begin{proof} 
% as required.  
    We apply Theorem \ref{t:dec_Tk} with $q \sim  \frac{\log s}{q \log q}$, that is, $q \sim \sqrt{\log s/\log \log s} \le s$. 
    Then  
    %%$\log s \gg q \log q$  and 
    in view of our assumption \eqref{cond:dec_Tk_c} we see that the condition $s q \log q \ll \log |A|$ takes place. 
    %The same assumption \eqref{cond:dec_Tk_c} implies that   $s\le \exp \left( \frac{\log |A|}{2\log \log |A|} \right)$. 
    Hence we obtain $\T^{\times}_s (C) \le |A|^{2s-\frac{c_* \sqrt{\log s}}{\sqrt{\log \log s}}}$
    and 
\[
   |B|^{2q-2s} \T^{+}_s (B) \le \T^{+}_q (B) \le |B|^{2q-\frac{c_* \sqrt{\log s}}{\sqrt{\log \log s}}} \,.
\]
   Thus we have obtained the required bound   \eqref{f:dec_Tk_c}.
%   This completes the proof. 
$\hfill\Box$   
\end{proof}

\bp 

Corollary \ref{c:dec_Tk} implies a result on additive/multiplicative Sidon sets in $\Z$, see the details of the proof in  \cite[Theorem 1.1]{Mudgal_Sidon}. 
Recall that a finite set $A\subseteq \Gr$ to be a $B^{+}_h [g]$ set if  for any $x\in \Gr$ one has $r_{hA} (x) \le g$ 
(and similarly for $B^{\times}_h [g]$). 
%Estimate \eqref{f:Mudgal_Sidon_2} is a quantitative analogue of the main result from \cite{BC}. 

\begin{corollary}
    Let $h$ be a positive integer, $A\subset \Z$ be a finite set, and let $B$ and $C$ be the largest $B^{+}_h [1]$ and $B^{\times}_h [1]$ sets in $A$ respectively. Then 
\[
    \max\{ |B|, |C| \} \gg |A|^{\eta_h/h} \,,
\]
    where $\eta_h \gg (\log h)^{1/2-o(1)}$.
    In particular, for any $A\subset \Z$
\begin{equation}\label{f:Mudgal_Sidon_2}
    |hA| + |A^h| \gg_h |A|^{(\log h)^{1/2-o(1)}} \,.
\end{equation} 
\label{c:Mudgal_Sidon}
\end{corollary}

%\bp 

Now we are ready to obtain a purely sum--product--type  result for dimensions. 

%Using the argument of \cite{ES} and \cite[Proposition 1.5]{ZP}, we obtain 

\begin{theorem}
    Let $A\subset \Z$ be an arbitrary finite set. Then 
\begin{equation}\label{f:sum-prod_dim_lower}
    \max\{ \dim^{+} (A), \dim^{\times} (A)\}
        \gg 
            \log |A| \cdot \frac{\sqrt{\log \log |A|}}{\sqrt{\log \log \log |A|}} \,.
\end{equation}
    On the other hand, there is $A\subset \Z$ such that 
\begin{equation}\label{f:sum-prod_dim}
    \max\{ \dim^{+} (A), \dim^{\times} (A)\}
        \ll 
            \log |A| \cdot \log \log |A| \,.
\end{equation}
\label{t:sum-prod_dim}
\end{theorem}
\begin{proof}
    Let $L=\log |A|$. 
    To get \eqref{f:sum-prod_dim_lower} we apply Corollary \ref{c:dec_Tk} and obtain pairwise disjoint sets $B$ and $C$ such that $A=B\bigsqcup C$ and estimate  \eqref{f:dec_Tk_c} takes place. 
    Here $s \ll L /\sqrt{\log L \log \log L}$. 
    Of course either $B$ or $C$ has size at least $|A|/2$ and suppose that this set is $B$.
    Thus we have very good upper bound \eqref{f:dec_Tk_c} for all $\T^{+}_s (B)$ and hence it is possible to apply estimate \eqref{f:T_k_dim_2} of Theorem \ref{t:T_k_dim}.
    More precisely,  as in formulae \eqref{tmp:03.05_1}, \eqref{tmp:03.05_1+} above, we have 
\[
    \D \log (8e \dim^{+} (B)/\D) \ll L \sqrt{\frac{\log \D}{\log \log \D}}
\]
    and hence for $m\ge \D \sim L/ \sqrt{\log L \cdot \log \log L}$ (one can assume that 
    $\dim^{+} (A) \ll L^2$, say, and whence $\log \dim^{+} (A) \sim \log L$ because otherwise there is nothing to prove) and further 
\[
    \dim^{+} (B) \gg L \sqrt{\frac{\log m}{\log \log m}} \,. 
\]
    Thus we see that 
\[
    \dim^{+} (A) \ge \dim^{+} (B) \gg \frac{L \sqrt{\log L} }{\sqrt{\log \log L}} 
    %\,. 
\]
    as required. 
%    Taking $s \sim L /\sqrt{\log L \log \log L}$, we obtain the required result. 

    Now to obtain \eqref{f:sum-prod_dim} we use the arguments from   \cite{ES} and \cite[Proposition 1.5]{ZP}.
    Namely, let $s,h \ge 2$ be integer parameters, which we will choose later.
    Put
\[
    A = \left\{ \prod_{i=1}^s p^{l_i}_i ~:~ l_i \in [h] \right\} \,,
\]
    where $p_i$, $i\in [s]$ are the first $s$ primes.  
    We have  $|A| = h^s$ and 
\[
    \max A = \exp (h \sum_{i=1}^s \log p_i)
    = \exp (O(hs \log s)) \,.
\]
    Trivially, $A \subseteq [\max A]$ and hence 
\begin{equation}\label{tmp:dim+}
    \dim^{+} (A) \ll \log (\max A) \ll hs \log s \,.
\end{equation}
    To estimate $\dim^{\times} (A)$, we use the same argument as in formula  \eqref{fi:pol_growth_1} of Proposition  \ref{p:pol_growth}.
    Indeed, thanks to the fact that $|A^n| \le (hn)^s$, we obtain 
\begin{equation}\label{tmp:dim*}
    \dim^{\times} (A) \ll n \exp (n^{-1} s \log (hn)) \ll s \log s \,.
\end{equation}
    Here we have chosen $n \sim s \log s$ and $h=2$. 
%    we have assumed that $h\ll n$, say.  
    Thus both bounds \eqref{tmp:dim+}, \eqref{tmp:dim*} are of the same quality and from $|A| = h^s = 2^s$, we obtain the required result. 
% as required.  
%    This completes the proof. 
$\hfill\Box$   
\end{proof}

\begin{remark}
    A similar construction as in Theorem \ref{t:sum-prod_dim} (see \cite[Proposition 1.5]{ZP}) shows that one cannot obtain something better than $\Omega(\log s/\log \log s)$ in estimate \eqref{f:dec_Tk_c}. 
\end{remark}

\section{Appendix}
\label{sec:appendix} 

For the convenience of the reader we give a short proof of an asymmetric version of the Balog--Szemer\'edi--Gowers Theorem with explicit dependence of the parameters on the quantity  $K := |A||B|^2 \E^{-1}(A,B)$.
Basically, we repeat the argument from the appendix of %\cite[Appendix]{rs_SL2}. 
\cite{rs_SL2}.

\begin{lemma}
    Let $A,B \subseteq \Gr$ be sets, 
    $|A|\ge |B|$ 
    and $\E(A,B)\ge |A||B|^2/K$. 
    Also, let $k$ be a positive   integer, 
    %%$k\le \log (\log |B|/\log K)$ 
    $k\le \frac{\log |B|}{\log \log |B|}$ 
    and $M := (|A|/|B|)^{1/k} K^{2^k/k}$.
    Then there is a set $H \subseteq \Gr$ such that $|H+H| \ll M^C |H|$ and for a certain $x\in \Gr$ one has $|B\cap (H+x)| \gg |B|/M^C$.
    Here $C>0$ is an absolute constant. 
\label{l:As_BszG}
\end{lemma}
\begin{proof} 
    %Let $r(x) =r_{B-B} (x)$. 
    We have
\[
    |A||B|^2/K \le \E(A,B) = 
    %\sum_{x\in A} \sum_{y} r_{B-B} (y) A(x+y) = 
    \sum_{x\in A} \sum_{y} r_{B-B}(y) A(x+y) \,,
\]
    and hence using the H\"older inequality several times, we get for any $j\le k$ 
\[
    |A| |B|^{2^j} K^{-2^{j-1}} \le \sum_{x\in A} \sum_{y} r_{2^{j-1} B - 2^{j-1} B} (y) A(x+y) \,.
\]
    Now applying 
    the Cauchy--Schwarz inequality  one more time, we see that 
\begin{equation}\label{tmp:05.05_1}
    %\frac{|B|}{|A|} 
    |A|^{-1}
    \cdot |B|^{2^{j+1}} K^{-2^{j}} 
    \le 
    |A|^2 \E^{-1} (A) |B|^{2^{j+1}} K^{-2^{j}} 
    \le 
        \T_{2^j} (B) \,.
\end{equation}
    Write $\T_l$ for $\T_l (B)$. 
    Then from the last formula, we have by the pigeonhole principle that there is $j\in [k]$ 
    %such that 
    with 
    $\T_{2^{j}} \ge |B|^{2^j} \T_{2^{j-1}}/M$. 
    Put $L = \log (8 |B|^{2^{j+1}-1} \T^{-1}_{2^{j}})$.
    In view of estimate \eqref{tmp:05.05_1}
    %as well as the bound  $\T_{2^{j}} \ge |B|^{2^j} \T_{2^{j-1}}/M$
    we see that $L\le \log (8 K^{2^j} |A|/|B|) \le \log (8 K^{2^k} |A|/|B|)$.
    %%Clearly, we can assume that $k\ll \log (\log |B|/\log K)$  because otherwise $M^C \ge |B|$ and hence there is nothing to prove. 
%%    Also, the same condition $M^C \ge |B|$ implies that $|A| \ge |B|^{1+k/C}$ and thus the set $A$ is large (actually, if $A$ and $B$ have comparable sizes, then everything  follows immediately from the usual Balog--Szemer\'edi--Gowers Theorem). 
%%    $k\le \frac{\log |B|}{\log \log |B|}$ 
%%    (actually, even the condition $k\le \frac{\log (|A|/|B|)}{\log \log (|A|/|B|)}$ works)
    One can easily see that $L\ll M^5$, say. 
    Indeed, it is sufficient to check that $k\log \log (|A|/|B|) \ll \log (|A|/|B|) + 2^k$ but if not then $k \ll \log \log \log (|A|/|B|)$ and hence $k\log \log (|A|/|B|) \ll \log (|A|/|B|)$ and it gives us a contradiction
    (one can assume that $|A|$ is much larger than $|B|$, actually, if $A$ and $B$ have comparable sizes, then everything  follows immediately from the usual Balog--Szemer\'edi--Gowers Theorem). 
    %equivalent to 
 %%   Using these two bounds one can check that $L\ll M^5$, say. 
    Now 	by the dyadic Dirichlet principle, our choice of the number $L$ and the H\"older inequality there is a number $\D >0$ and a set $P = \{ x \in \Gr ~:~ \D <  r_{2^{j-1} B} (x) \le 2 \D \}$ such that
	\begin{equation}\label{tmp:19.09_3}
	2L^4 \D^4 \E(P) \ge 
	\T_{2^{j}} \ge \frac{|B|^{2^j} \T_{2^{j-1}}}{M} 
	\ge \frac{(\D |P|)^2 \D^2 |P|}{M} \,,
	\end{equation}
	and hence $\E(P) \ge |P|^3/(2L^4 M) := |P|^3/Q$. 
	%By 
	Applying the 
	Balog--Szemer\'edi--Gowers Theorem \cite[Theorem 32]{TV}, we find a set $H\subseteq P$ such that $|H|\gg |P|/Q^C$ and  $|H+H|\ll Q^C |H|$, where $C$ is an absolute constant.
	Using the definition of the set $P$, we have 
\begin{equation}\label{tmp:19.09_3-}
    \D |P| Q^{-C} \ll \D |H| < \sum_{x\in H} r_{2^{j-1} B} (x) \le |B|^{2^{j-1}-1} \cdot \max_{x\in \Gr} |B\cap (H+x)| \,.
\end{equation}
    Now applying \eqref{tmp:19.09_3}, we get
\[
    \frac{|B|^{2^j} \T_{2^{j-1}}}{M} 
    \le 
    2 L^4 \D^4 |P|^3 \le 2 L^4 \T_{2^{j-1}} (\D |P|)^2 \,,
\]
    and hence $\D |P| \ge 2^{-1} L^{-2} |B|^{2^{j-1}}$.
    Combining  the last bound with \eqref{tmp:19.09_3-}, we find $x\in \Gr$ such that $|B\cap (H+x)| \gg |B|/Q^{C+1}$.
    It remains to recall the definitions of $Q$ and $L$.  
% as required.  
    This completes the proof. 
$\hfill\Box$   
\end{proof}

\bigskip

\noindent{I.D.~Shkredov\\
Steklov Mathematical Institute,\\
ul. Gubkina, 8, Moscow, Russia, 119991}
\\
%and
%\\
%IITP RAS,  \\
%Bolshoy Karetny per. 19, Moscow, Russia, 127994\\
%and 
%\\
%MIPT, \\ 
%Institutskii per. 9, Dolgoprudnii, Russia, 141701\\
{\tt ilya.shkredov@gmail.com}

\end{document}